\documentclass[12pt]{amsart}
\usepackage{amsrefs}   
\usepackage{fullpage}
\usepackage{amssymb}
\title[]{The $G$-stable rank for tensors}
\author[]{Harm Derksen}
\thanks{The author was supported by NSF grant IIS 1837985.}

\newtheorem{theorem}{Theorem}[section]
\newtheorem{proposition}[theorem]{Proposition}
\newtheorem{conjecture}[theorem]{Conjecture}
\newtheorem{lemma}[theorem]{Lemma}
\newtheorem{corollary}[theorem]{Corollary}
\theoremstyle{definition}
\newtheorem{definition}[theorem]{Definition}

\newtheorem{example}[theorem]{Example}

\newcommand{\Trace}{\operatorname{Tr}}

\newcommand{\ncrk}{\operatorname{ncrk}}

\newcommand{\SL}{\operatorname{SL}}

\newcommand{\C}{{\mathbb C}}
\newcommand{\R}{{\mathbb R}}
\newcommand{\Z}{{\mathbb Z}}

\newcommand{\Q}{{\mathbb Q}}
\newcommand{\F}{{\mathbb F}}
\newcommand{\LS}{K(\!(t)\!)}
\newcommand{\PS}{K[[t]]}
\newcommand{\supp}{\operatorname{supp}}
\newcommand{\Gm}{{\mathbb G}_m}
\newcommand{\Spec}{\operatorname{Spec}}

\newcommand{\LP}{{\bf LP}}

\newcommand{\End}{\operatorname{End}}
\newcommand{\rk}{\operatorname{rk}}
\newcommand{\GL}{\operatorname{GL}}
\newcommand{\vv}{\operatorname{val}_t}
\newcommand{\slicerk}{\operatorname{srk}}
\newcommand{\srk}{\operatorname{srk}}
\newcommand{\brk}{\operatorname{brk}}

\begin{document}

\maketitle
\setcounter{tocdepth}{1}

\begin{abstract}
    We introduce the $G$-stable rank of a higher order tensors over perfect fields.
    The $G$-stable rank is related to the Hilbert-Mumford criterion for stability in Geometric Invariant Theory. We will relate the $G$-stable rank to the tensor rank and slice rank. For numerical applications, we express the $G$-stable rank as a solution to an optimization problem. Over the field $\F_3$ we discuss an application to the Cap Set Problem.
\end{abstract}
\tableofcontents
\section{Introduction}
\subsection{Ranks of tensors}
We will introduce the $G$-stable rank for tensors, describe its properties and relate it to other notions for the rank of a tensor, such as the {\em tensor rank}, {\em border rank}, {\em slice rank} and {\em non-commutative rank}. Suppose that $K$ is a field,  $V_1,V_2,\dots,V_d$ are finite dimensional $K$-vector spaces  and $V=V_1\otimes V_2\otimes \cdots \otimes V_d$ is the tensor product. All tensor products are assumed to be over the field $K$ unless stated otherwise.
The definition of tensor rank goes back to Hitchcock \cite{Hitchcock27a,Hitchcock27b}:
\begin{definition}
The rank $\rk(v)$ of a tensor $v\in V$ is the smallest nonnegative integer $r$ such that we can write 
$v=\sum_{i=1}^r v_{i,1}\otimes v_{i,2}\otimes \cdots\otimes v_{i,d}$ 
with $v_{i,j}\in V_j$ for all $i$ and $j$.  
\end{definition}
There are many applications of the tensor rank and the related concept of CP-decomposition (see~\cite{BaderKolda09} for a survey). 
For $d=2$, tensor rank coincides with  matrix rank. Computing the tensor rank is NP-hard \cite{Hastad89,Hastad90},
and tensor rank is ill-behaved. For example,
the set $X(\rk,r)\subseteq V$ of all tensors of rank $\leq r$
is not always Zariski closed. The {\em border rank} $\brk(v)$ of a tensor $v$ is the smallest positive integer $r$ such that $v\in X(\rk,r)$ (see \cite{BCS97, Landsberg12}). The  {\em slice rank} of a tensor was introduced by Terence Tao (see~\cite{TaoSawin16,BCCGNSU17}). 
\begin{definition}
A non-zero tensor $v\in V$ has slice rank $1$ if it is contained in $$V_1\otimes \cdots \otimes V_{i-1}\otimes w\otimes V_{i+1}\otimes \cdots\otimes V_d$$ 
for some $i$ and some $w\in V_i$. The slice rank $\slicerk(v)$ of an arbitrary tensor $v\in V$ is the smallest nonnegative integer $r$
such that $v$ is the sum of $r$ tensors with slice rank $1$.
\end{definition}

 \subsection{The definition of the $G$-stable rank}

 We will now define the $G$-stable rank. It was noted in \cite{BCCGNSU17} that the slice-rank is closely related the notion of stability in Geometric Invariant Theory (see~\cite{MFK94}). The authors
 also introduce the {\em instability} of a tensor and relate it to the slice rank. The instability of a tensor does not behave like a rank function, but it is closely related to the $G$-stable rank. We will define the $G$-stable rank in terms of degenerations and power series. It can also be defined in terms 1-parameter subgroup using the Hilbert-Mumford criterion in Geometric Invariant Theory (see Theorem~\ref{theo:1PSG}). The Hilbert-Mumford criterion is often formulated when working over an algebraically closed field $K$. Kempf showed in \cite{Kempf78} that the Hilbert-Mumford criterion still applies when working of a perfect field $K$. For this reason,  we will assume that $K$ is a perfect field for the remainder of the paper.

To define the $G$-stable rank, we need to introduce the ring $\PS$ of formal power series in $t$ and its quotient field $\LS$ of formal Laurent series. 
The $t$-valuation of a series $a(t)\in \LS$ is the smallest integer $d$ such that $a(t)=t^db(t)$ with $b(t)\in K[[t]]$. 
By convention, $\vv(0)=\infty$. 
 If $W$ is a $K$-vector space and $v(t)\in \LS\otimes W$ then we define
$$\vv(v(t))=\min\{d\mid  v(t)=t^dw(t)\mbox{ and $w(t)\in K[[t]]\otimes W$}\}.$$
We say that $v(t)$ has no poles when $\vv(v(t))\geq 0$, which is equivalent to $v(t)\in K[[t]]\otimes W$. In that case
we say that $\lim_{t\to 0} v(t)$ exists, and is equal to $v(0)\in W$.

The group $\GL(W,\LS)$ will denote the group of $\LS$-linear endomorphisms of the space $\LS\otimes_K W$.
We may view $\GL(W,\LS)$ as a subset of $\LS\otimes_K\End(W)$. If $W=K^n$ then $K(t)\otimes_K W\cong K((t))^n$ and we can identify $\GL(W,\LS)$ with the set of $n\times n$ matrices with entries in the field $\LS$.
If $R\subseteq \LS$ is a $K$-subalgebra of $\LS$ (such as $R=\PS$,
$R=K[t,t^{-1}]$ or $R=K[t]$), 
then $\GL(W,R)$ is the intersection of $\GL(W,\LS)$ with $R\otimes_K \End(W)$ in $\LS\otimes_K \End(W)$. Note that the
inverse of an element in $\GL(W,R)$ lies in $\GL(W,\LS)$, but not necessarily in $\GL(W,R)$. If $W=K^n$, then $\GL(W,R)$ is the set of $n\times n$ matrices with entries in $R$ that, viewed as a matrix with entries in $\LS$, are invertible.

We consider the action of the group $G=\GL(V_1)\times \GL(V_2)\times \cdots\times \GL(V_d)$ on the tensor product space $V=V_1\otimes V_2\otimes \cdots\otimes V_d$. For any $K$-subalgebra $R\subseteq \LS$, we define
$$G(R)=\GL(V_1,R)\times \cdots\times \GL(V_d,R).$$
The group $G(\LS)$ acts on $\LS\otimes V$. 

For any weight $\alpha=(\alpha_1,\alpha_2,\dots,\alpha_d)\in \R^d_{>0}$
we will have a notion of $G$-stable rank, but the case $\alpha=(1,1,\dots,1)$ will be of particular interest. Suppose that $g(t)\in G(\PS)$, $v\in V$
and $\vv(g(t)\cdot v)>0$. We consider the slope
\begin{equation}\label{eq:inf}
\mu_\alpha(g(t),v)=\frac{\sum_{i=1}^d \alpha_i \vv( \det g_i(t))}{\vv(g(t)\cdot v)}.
\end{equation}
Heuristically, the denominator in the slope measures how fast $g(t)\cdot v$ goes to $0$ as $t\to 0$. The numerator measures how fast the eigenvalues of $g_1(t),g_2(t),\dots,g_d(t)$ go to $0$ as $t\to 0$. A small slope means that $v$ is very unstable in the sense that $g(t)\cdot v$ goes to $0$ quickly, while, on average,  the eigenvalues of $g_i(t)$ go to 0 slowly.
\begin{definition}
The $G$-stable $\alpha$-rank $\rk^G_{\alpha}(v)$ of $v$
as the infimum of all $\mu_{\alpha}(g(t),v)$ where $g(t)\in G(\PS)$ and $\vv(g(t)\cdot v)>0$.
If $\alpha=(1,1,\dots,1)$, then we may write $\rk^G$ instead of $\rk^G_{\alpha}$.
\end{definition}
Using a $K$-rational version of the Hilbert-Mumford (\cite{Hilbert93,MFK94}) criterion by Kempf~\cite{Kempf78}, we will show that for computing the $G$-stable $\alpha$-rank, one only has to consider $g(t)$ that are $1$-parameter subgroups of $G$ without poles (Theorem~\ref{theo:1PSG}).
In this context, $g(t)\in G(K[t])$ is a $1$-parameter subgroup if for every $i$ we can choose a basis of $V_i$ such that the matrix of $g(t)$ is diagonal
and each diagonal entry of that matrix is a nonnegative power of $t$.

We denote the standard basis vectors in $K^n$ by $[1],[2],\dots,[n]$, and we abbreviate a tensor $[i_1]\otimes [i_2]\otimes \cdots\otimes [i_d]$
by $[i_1,i_2,\dots,i_d]$.
\begin{example}\label{example:2}
Suppose that $V_1=V_2=V_3=K^2$, and $v=[2,1,1]+[1,2,1]+[1,1,2]$.
We take $g(t)=(g_1(t),g_2(t),g_3(t))$
with 
$$
g_1(t)=g_2(t)=g_3(t)=\begin{pmatrix} t & 0\\ 0 & 1\end{pmatrix}.
$$
We have $g(t)\cdot v=t^2 v$, $\det(g_i(t))=t$, and
$$
\mu(g(t),v)=\mu_{(1,1,1)}(g(t),v)=\frac{\vv(\det g_1(t))+\vv(\det g_2(t))+\vv(\det g_3(t))}{\vv(g(t)\cdot v)}=\textstyle \frac{1+1+1}{2}=\frac{3}{2}.
$$
This shows that $\rk^G(v)\leq \frac{3}{2}$. One can show 
that $\rk^G(v)=\frac{3}{2}$ (see Example~\ref{example:3} and Example~\ref{example:222tensor}).
\end{example}
\subsection{Properties of the $G$-stable rank}
If $v$ is a rank $1$ tensor, then we have $\rk_\alpha^G(v)=\min\{\alpha_1,\dots,\alpha_d\}$
and $\rk^G(v)=1$ (Lemma~\ref{lem:Granklb}).
The $G$-stable rank is related to other notions of rank. We have (see Corollary~\ref{cor:GstableSlice} and Proposition~\ref{prop:GstableSlice3})
$$
\frac{2\slicerk(v)}{d}\leq \rk^G(v)\leq \slicerk(v)\leq \brk(v)\leq \rk(v).
$$
This implies that for $d=2$, the $G$-stable rank, the slice rank and the matrix rank coincide.

The tensor rank depends on the field one is working over. For example, the tensor $[1,1,1]-[1,2,2]-[2,1,2]-[2,2,1]$
has rank $3$ as a tensor in $\R^{2\times 2\times 2}$ but rank $2$ when viewed as a tensor in $\C^{2\times 2\times 2}$.
Although it is not clear from the definition, the $G$-stable rank does not change when passing to a field extension of $K$ (see Theorem~\ref{theo:FieldIndependent}).

Another nice property of the $G$-stable rank is that the border rank phenomenon does not happen and the
set $X(\rk^G_\alpha,r)$ of all tensors $v$ with $\rk^G_\alpha(v)\leq r$ is Zariski closed (Theorem~\ref{theo:ZariskiClosed}). Tao proved a  similar result for the slice rank~\cite{Tao16}, and this implies that $\slicerk(v)\leq \brk(v)$ for all tensors $v$.

Like other rank notions, the $G$-stable rank satisfies the triangle inequality: $\rk^G_\alpha(v+w)\leq \rk^G_\alpha(v)+\rk^G_\alpha(w)$ (see Proposition~\ref{prop:triangle}).
 If $v\in V_1\otimes V_2\otimes \cdots \otimes V_d$
and $w\in W_1\otimes W_2\otimes \cdots \otimes W_d$ then the direct sum of $v$ and $w$, viewed as
$$
\begin{pmatrix}v \\ w
\end{pmatrix}\in 
\begin{array}{c}
V_1\otimes V_2\otimes \cdots \otimes V_d\\ \oplus\\ W_1\otimes W_2\otimes \cdots\otimes W_d
\end{array}
\subseteq V\boxplus W:=
\begin{pmatrix}
V_1\\ \oplus \\ W_1\end{pmatrix}\otimes 
\begin{pmatrix}
V_2\\ \oplus \\ W_2\end{pmatrix}\otimes \cdots\otimes
\begin{pmatrix}
V_d\\ \oplus \\ W_d\end{pmatrix}
$$
will be denoted by $v\boxplus w$.
 (We will
use the notation $v\boxplus w$ and $V\boxplus W$ rather than the more common notation $v\oplus w$ and $V\oplus W$ to emphasize that this direct sum
is a ``vertical'' operation, i.e., the sum $V_i\oplus W_i$ is taken within each tensor factor.)
The $G$-stable rank is additive (Proposition~\ref{prop:additive}): $\rk^G_{\alpha}(v\boxplus w)=\rk^G_{\alpha}(v)+\rk^G_{\alpha} (w)$. In particular, if
\begin{multline*}
v=[1,1,\dots,1]+[2,2,\dots,2]+\cdots+[r,r,\dots,r]=\\=\underbrace{[1,1,\dots,1]\boxplus [1,1,\dots,1]\boxplus \cdots\boxplus [1,1,\dots,1]}_r\in \underbrace{K^{r}\otimes K^r\otimes \cdots\otimes K^{r}}_d,
\end{multline*}
then $\rk^G_\alpha(v)=r\rk^G_\alpha([1,1,\dots,1])=r\min\{\alpha_1,\dots,\alpha_d\}$ and $\rk^G(v)=r$.
Strassen conjectured in \cite{Strassen73} that tensor rank is additive when $K$ is infinite, but Shitov recently gave a counterexample to
this long standing conjecture (see~\cite{Shitov19}).

If $v\in V_1\otimes V_2\otimes \cdots\otimes V_d$ and $w\in W_1\otimes W_2\otimes \cdots\otimes W_e$, then we can
form the ``horizontal'' tensor product $v\otimes w\in V_1\otimes \cdots\otimes V_d\otimes W_1\otimes \cdots \otimes W_e$. It is clear that
$\rk(v\otimes w)\leq \rk(v)\rk(w)$. It was recently shown in \cite{CGJ19} that we do not always have equality. The $G$-stable rank behaves quite differently 
for the horizontal tensor product. We have $\rk^G_{\alpha,\beta}(v\otimes w)=\min\{\rk^G_{\alpha}(v),\rk^{G}_{\beta}(w)\}$ (see Proposition~\ref{prop:horizontal}).
If $d=e$ then there is another way of forming a tensor product.  The tensor product  $v\otimes w$ viewed as
$$
\begin{array}{c}v \\ \otimes\\ w
\end{array}\in 
\begin{array}{c}
V_1\otimes V_2\otimes \cdots \otimes V_d\\ \otimes \\ W_1\otimes W_2\otimes \cdots\otimes W_d
\end{array}\subseteq 
\begin{pmatrix}
V_1\\ \otimes \\ W_1\end{pmatrix}\otimes 
\begin{pmatrix}
V_2\\ \otimes \\ W_2\end{pmatrix}\otimes \cdots\otimes
\begin{pmatrix}
V_d\\ \otimes \\ W_d\end{pmatrix}
$$
will be denoted by $v\boxtimes w$.
We will refer to this operation as a {\em vertical} tensor product or a {\em Kronecker} tensor product. It is clear that $\rk(v\boxtimes w)\leq \rk(v\otimes w)$.
It has long been known that $\rk(v\boxtimes w)$ can be smaller than $\rk(v)\rk(w)$. For example, if $v_1=[1,1,1]+[2,2,1]$, $v_2=[1,1,1]+[2,1,2]$ and $v_3=[1,1,1]+[2,2,1]$
then $v_1\boxtimes v_2\boxtimes v_3$ is the matrix multiplication tensor for $2\times 2$ matrices which has rank $7$ (\cite{Strassen69}),
so $7=\rk(v_1\boxtimes v_2\boxtimes v_3)<\rk(v_1)\rk(v_2)\rk(v_3)=2^3$.
 If $K$ has characteristic 0,  
then we have $\rk^G_{\alpha\beta}(v\boxtimes w)\geq \rk^G_{\alpha}(v)\rk^G_{\beta}(v)$ (Theorem~\ref{theo:supermultiplicative}).
We conjecture that this inequality is also true when $K$ is a perfect field of positive characteristic. The slice rank does not behave as nicely with respect to
vertical tensor product and $\srk(v\boxtimes w)$ could be larger or smaller than $\srk(v)\srk(w)$ (see \cite[Example 5.2]{CVZ17}).

\subsection{$G$-stable rank for complex tensors}
If $K=\C$, then the $G$-stable rank can be computed in a different way. For a finite dimensional complex Hilbert space, we will denote the Hermian form by $\langle\cdot,\cdot\rangle$
and the $\ell_2$ norm (or Frobenius norm) by $\|v\|=\sqrt{\langle v,v\rangle}$.
Suppose that $V_1,V_2,\dots,V_d$ are finite dimensional Hilbert spaces, which
makes $V$ into a Hilbert space. If $A$ is a linear map between finite dimensional Hilbert spaces, then its spectral norm $\|A\|_\sigma$
is the operator norm $\|A\|_\sigma=\max_{v\neq 0}\frac{\|Av\|}{\|v\|}$, which is also the largest singular value of $A$.

 For a tensor $v\in V$, let $\Phi_i(v):(V_1\cdots \otimes \widehat{V}_i\otimes \cdots \otimes V_d)^\star\to V_i$
be the $i$-th flattening. Then the $G$-stable $\alpha$-rank of a tensor $v\in V$ is equal to 
\begin{equation}\label{eq:rkG_C}
\rk^G_\alpha(v)=\sup_{g\in G} \min_i \frac{\alpha_i\|g\cdot v\|^2}{\|\Phi_i(g\cdot v)\|^2_\sigma}
\end{equation}
(see Theorem~\ref{theo:Cstablerank}).
\begin{example}\label{example:3}
Consider again the example $v=[2,1,1]+[1,2,1]+[1,1,2]\in K^{2\times 2\times 2}$ as in Example~\ref{example:2}, but now we will work over $K=\C$. We have $\|v\|=\sqrt{3}$.
The first flattening of $v$ is equal to 
$$
\Phi_1(v)=\left(\begin{array}{cc|cc}
0 & 1 & 1 & 0\\
1 & 0 & 0 & 0
\end{array}\right)
$$
which has singular values $1$ and $\sqrt{2}$. So $\|\Phi_1(v)\|_\sigma=\sqrt{2}$. 
By symmetry, we also have $\|\Phi_2(v)\|_\sigma=\|\Phi_3(v)\|_\sigma=\sqrt{2}$. It follows that
$$
\rk^G(v)=\sup_{g\in G} \min_i \frac{\|g\cdot v\|^2}{\|\Phi_i(g\cdot v)\|^2_\sigma}\geq \min_i\frac{\|v\|^2}{\|\Phi_i(v)\|^2_\sigma}=\textstyle\frac{3}{2}.
$$
\end{example}

\subsection{The Cap Set Problem}\label{sec:capset}
We say that a subset $S$ of an abelian group $A$ does not contain an arithmetic progression (of length 3) if there are no distinct elements $x,y,z\in S$
with $x+z=2y$. For an abelian group $A$, let $r_3(A)$ be the largest cardinality of a subset $S\subseteq A$ without an arithmetic progression.
Finding upper and lower bounds for $r_3(A)$ has been studied extensively in number theory. For the group $A=(\Z/3\Z)^n\cong \F_3^n$ this is known as the Cap Set Problem.
Brown and Buhler \cite{BrownBuhler82} showed that $r_3(\F_3^n)=o(3^n)$ and this was later improved to  $r_3(\F_3^n)=O(3^n/n)$ by Meshulam \cite{Meshulam95} and to $o(3^n/n^{1+\varepsilon})$ by Bateman and Katz \cite{BatemanKatz12}.
Using the polynomial method of Croot,  Lev and Pach \cite{CLP17}, who showed that $r_3((\Z/4\Z)^n)=o(c^n)$ for some $c<4$, Ellenberg and Gijswijt showed in \cite{EllenbergGijswijt17} that $r_3(\F_3^n)\leq 3\theta^n=o(2.756^n)$,
where $\theta <2.756$.  We also have a lower bound $r_3(\F_3^n)=\omega(2.21^n)$ by Edel.  The bound (and the proof) of Ellenberg and Gijswijt is also valid for tri-colored sum-free sets
for which an asymptotic lower bound $\omega(\theta^n)$ was given by Kleinberg, Sawin and Speyer \cite{KSS18}. So for tri-colored sum-free sets, the upper and lower bound have the same exponential growth.

Tao noted that the Ellenberg-Gijswijt proof can be nicely presented using the concept of slice rank. A key idea is to prove the inequality
$r_3(\F_3^n)\leq \slicerk(u^{\boxtimes n})$ where
$$
u=\sum_{\scriptstyle i,j,k\in \Z/3\Z\atop \scriptstyle i+j+k=0} [i,j,k]\in \F_3^{3\times 3\times 3}
$$
and to combine this with asymptotic estimates for the slice rank. 
We will show that $r_3(\F_3^n)\leq \rk^G(u^{\boxtimes n})\leq  \srk(u^{\boxtimes n})$. Using the $G$-stable rank, we get better upper bounds for the cardinality of a cap set (or a tri-colored sum-free set).
Below is a table of the upper bounds we get for $n\leq 20$.
$$
\begin{array}{||r|r||r|r||r|r||r|r||}
\hline
n & \mbox{upper bound} & n & \mbox{upper bound} & n & \mbox{upper bound} & n & \mbox{upper bound}\\ \hline\hline
1 & 2 & 6 & 274 & 11 & 37477 & 16 & 5235597 \\ \hline
2 & 6 &7 & 722 &   12 &  100296 &  17 & 14316784  \\ \hline
3 & 15 & 8 & 1957 & 13 & 266997 &  18 & 38685141\\ \hline
4 & 39 & 9 & 5193 & 14 & 728661 & 19 & 103504935\\ \hline
5 & 105  & 10 & 13770 & 15 & 1961103 &  20 & 283466139\\ \hline\hline
\end{array}
$$

\section{The $G$-stable rank and the Hilbert-Mumford criterion}
\subsection{The Hilbert-Mumford criterion}
We will discuss the $K$-rational version of the Hilbert-Mumford criterion by Kempf \cite{Kempf78}. 
We remind the reader that the base field $K$ is assumed
to be perfect. Suppose that $G$ is a connected reductive algebraic group over a field $K$, $X$ is a separated $K$-scheme of finite type
and $G\times X\to X$ is a $G$-action that is also a morphism of schemes over $K$. 
The multiplicative group is defined as $\Gm=\Spec K[t,t^{-1}]$.  A $1$-parameter subgroup of $G$
is a homomorphism $\lambda:\Gm\to G$ of algebraic groups. We say that this $1$-parameter subgroup of $G$ is $K$-rational
if the homomorphism is a morphism of algebraic varieties defined over $K$.
In the case where $K$ is finite, we caution the reader that the set $G(K)$ of $K$ rational points in $G$ is finite and
may not be Zariski dense in the algebraic group $G$. If $x\in X(K)$ is a $K$-rational point of $X$, then $G\cdot x$ denotes 
a subscheme of $X$ which is not necessarily Zariski closed (even if $G(K)$ is finite). The Zariski closure $\overline{G\cdot x}$ is
a closed subscheme of $X$.

\begin{theorem}[{\cite[Corollary 4.3]{Kempf78}}]\label{theo:Kempf}
Suppose that $x\in X(K)$ is a $K$-rational point, $S\subseteq X$ is a $G$-invariant closed subscheme of $X$
such that $\overline{G\cdot x}\cap S\neq \emptyset$, Then there exists a $K$-rational $1$-parameter subgroup $\lambda:\Gm\to G$
such that $\lim_{t\to 0}\lambda(t)\cdot x=y$ for some $y\in S(K)$.
\end{theorem}
In our situation, $X=V$ is a $K$-vector space which is a representation of $G$, and $S=\{0\}$. A vector $v\in V$ is called
$G$-semi-stable if $\overline{G\cdot v}$ does not contain $0$. Now Theorem~\ref{theo:Kempf} implies:

\begin{corollary}
If $G$ is a connected reductive algebraic group, $v\in V$ and $0\in \overline{G\cdot v}$ then there exists a $K$-rational $1$-parameter subgroup $\lambda:\Gm\to G$ such that $\lim_{t\to 0} \lambda(t)\cdot v=0$.
\end{corollary}
 A $1$-parameter subgroup of $\GL_n$
is of the form
$$
\lambda(t)=C\begin{pmatrix}
t^{x(1)} & & & \\
 & t^{x(2)} & & \\
 & & \ddots & \\
 & & & t^{x(n)}
 \end{pmatrix}C^{-1}
 $$
 with $C\in \GL_n$ and $x(1),x(2),\dots,x(n)\in \Z$. In particular, we can view $\lambda$ as an element of $\GL_n(K[t,t^{-1}])$
 where $K[t,t^{-1}]\subseteq \LS$ is the ring of Laurent polynomials.
 If $v=(v_1\ v_2\ \cdots\ v_n)^t\in K^n$ then  $\lim_{t\to 0}\lambda(t)\cdot v=0$ if for all $i$, we have $v_i=0$ or $x(i)>0$.
 We will take $V=V_1\otimes V_2\otimes \cdots\otimes V_d$ and $G=\GL(V_1)\times \GL(V_2)\times \cdots\times \GL(V_d)$.
 A $1$-parameter subgroup of $G$ is of the form $(\lambda_1(t),\lambda_2(t),\dots,\lambda_d(t))$ where
 $\lambda_i(t):\Gm\to \GL(V_i)$ is a $1$-parameter subgroup for all $i$.

For an integer vector $\alpha=(\alpha_1,\alpha_2,\dots,\alpha_d)\in \Z^d$ we define
a homomorphism of algebraic groups $\det^{\alpha}:G\to \Gm$ by $(A_1,\dots,A_d)\mapsto \prod_{i=1}^d \det(A_i)^{\alpha_i}$.
This homomorphism corresponds to a $1$-dimensional representation of $G$, which we will also denote by $\det^{\alpha}$.
We will now relate the $G$-stable rank to semi-stability in Geometric Invariant Theory.
\begin{proposition}\label{prop:stable}
Suppose that $\beta\in \Q_{>0}^d$, $p$ is a nonnegative integer and $q$ is a positive integer with $q\beta\in\Z^n$.
We define a representation $W$ by
$$
\textstyle W=\left(V^{\otimes p}\otimes \det^{-q\beta} \right)\oplus V_1^{n_1}\oplus V_2^{n_2}\oplus \cdots\oplus V_d^{n_d}.
$$
and choose $u_i\in V_i^{n_i}\cong K^{n_i\times n_i}$ of maximal rank $n_i$ for every $i$.
Then we have $\rk^G_\beta(v)\geq \frac{p}{q}$ if and only if $w=(v^{\otimes p}\otimes 1,u_1,\dots,u_d)$ is $G$-semi-stable.
\end{proposition}
\begin{proof}
Suppose that $\rk^G_{\beta}(v)<\frac{p}{q}$. Then there exists $g(t)=(g_1(t),\dots,g_d(t))\in G(K[[t]])$ with 
$$\textstyle\vv(g(t)\cdot (v^{\otimes p}\otimes 1))=p\vv(g(t)\cdot v)-\sum_{i=1}^dq\beta_i\vv(g_i(t))>0.
$$
The limit
$\lim_{t\to 0} g(t)\cdot w=(0,g(0)\cdot u)=(0,g(0)\cdot u_1,\dots,g(0)\cdot u_d)$
lies in the closure of the orbit $G\cdot w$. Since $0$ lies in the orbit closure of $(0,g(0)\cdot u)$,
it also lies in the orbit closure of $w$. We conclude that $w$ is not $G$-semistable.

Now suppose that $w$ is not $G$-semistable. By the Hilbert-Mumford criterion, there exists  a 1-parameter subgroup $\lambda(t)=(\lambda_1(t),\dots,\lambda_d(t))\in G(K[t,t^{-1}])$ of $G$ such that $\lim_{t\to 0}\lambda(t)\cdot w=0$.
This implies that $\lim_{t\to 0} \lambda_i(t)\cdot u_i=0$.
Since $u_i$ has maximal rank, we get $\lim_{t\to 0}\lambda_i(t)=0$ and  $\lambda_i(t)\in \GL(V_i,K[t])$. So we have $\lambda(t)\in G(K[t])\subseteq G(K[[t]])$.
We also get
$$
\textstyle 0<\vv(\lambda(t)\cdot (v^{\otimes p}\otimes 1))=p\vv(\lambda(t)\cdot v)-\sum_{i=1}^dq\beta_i\vv(\lambda_i(t))
$$
and therefore
$$
\mu_\beta^G(v)=\frac{\sum_{i=1}^d \beta_i\vv(\lambda_i(t))}{\vv(\lambda(t)\cdot v}<\textstyle\frac{p}{q}.
$$
We conclude that $\rk^G_\beta(v)<\frac{p}{q}$.

\end{proof}

\begin{theorem}\label{theo:1PSG}
If $\alpha\in \R_{>0}^d$, then the $G$-stable rank $\rk^G_\alpha(v)$ is the infimum of $\mu_\alpha(\lambda(t),v)$
where $\lambda(t)\in G(K[t])$ is a $1$-parameter subgroup of $G$ and $\vv(\lambda(t)\cdot v)>0$.
\end{theorem}
\begin{proof}
Asume that  $\rk^G_{\alpha}(v)<r$ for some rational number $r$. There exists a $\beta\in \Q_{>0}^d$
with $\beta-\alpha\in \R_{>0}^d$ and $\rk^G_{\beta}(v)<r$. We can write $r=\frac{p}{q}$ where $p$ and $q$ are positive integers
such that $q\beta\in \Z^d$. By Proposition~\ref{prop:stable}, $w$ is not $G$-semistable and from the proof of Proposition~\ref{prop:stable} follow that
 there exists a $1$-parameter subgroup $\lambda(t)\in G(K[t])$
such that $\mu_{\alpha}(\lambda(t),v)\leq \mu_\beta(\lambda(t),v)<r$.
This shows that even if $\lambda(t)\in G(K[t])$ is a $1$-parameter subgroup of $G$, $\mu_{\alpha}(\lambda(t),v)$ can get arbitrarily close to $\rk^G_{\alpha}(v)$.
\end{proof}
\subsection{The relation between $G$-stable rank and $\SL$-stability}
First we prove that the $G$-stable rank does not change when we extend the field.
\begin{theorem}\label{theo:FieldIndependent}
Suppose that $v\in V=V_1\otimes_KV_2\otimes_K\otimes \cdots\otimes_K V_d$ where $V_1,V_2,\dots,V_d$ are finite dimensional $K$-vector spaces, 
and $\overline{v}=1\otimes v\in \overline{V}=L\otimes_K V\cong \overline{V}_1\otimes_L\overline{V}_2\otimes_L\otimes \cdots\otimes_L \overline{V}_d$
with $\overline{V}_i=L\otimes_K V_i$ for all $i$. Then we have $\rk_\alpha^G(v)=\rk_\alpha^G(\overline{v})$. In other words, the $G$-stable rank does not change
under base field extension.
\end{theorem}
\begin{proof}
If $\beta\in \Q_{>0}^d$ then we can follow the set up in Proposition~\ref{prop:stable}, where $p,q\in \Z$, $p\geq 0$, $q>0$ and $q\beta\in \Z^d$. We choose $u_i\in V_i^{n_i}$ invertible for all $i$, and define
$$
w=(v^{\otimes p}\otimes 1,u_1,\dots,u_d)\in \textstyle W=\left(V^{\otimes p}\otimes_K \det^{-q\beta} \right)\oplus V_1^{n_1}\oplus V_2^{n_2}\oplus \cdots\oplus V_d^{n_d}.
$$
Using the base field extension, we get
$$
\overline{w}=(\overline{v}\otimes 1,\overline{u}_1,\dots,\overline{u}_d)\in \textstyle L\otimes_KW=\left(\overline{V}^{\otimes p}\otimes_L \det^{-q\beta} \right)\oplus \overline{V}_1^{n_1}\oplus \overline{V}_2^{n_2}\oplus \cdots\oplus \overline{V}_d^{n_d}.
$$
Now $G$-semistability does not chance after base field extension. So $w$ is $G$-semistable if and only if $\overline{w}$ is $G$-semistable. So we have
$$
\textstyle\rk_\beta^G(w)\geq \frac{p}{q}\Leftrightarrow \mbox{$w$ is $G$-semistable}\Leftrightarrow\mbox{$\overline{w}$ is $G$-semi-stable}\Leftrightarrow \rk_\beta^G(\overline{w})\geq \frac{p}{q}.
$$
This proves that $\rk_\beta^G(w)=\rk_\beta^G(\overline{w})$.
Since $\rk_\alpha^G(w)$ is the supremum of $\rk_\beta^G(w)$ over all $\beta\in \Q_{>0}^d$ with $\beta\leq \alpha$,
we also get $\rk_\alpha^G(w)=\rk_\alpha^G(w)$ for all $\alpha\in \R_{>0}^d$.

\end{proof}

\begin{proposition}
Suppose that $\alpha=(\frac{1}{n_1},\frac{1}{n_2},\dots,\frac{1}{n_d})$ where $n_i=\dim V_i$. For $v\in V=V_1\otimes V_2\otimes \cdots\otimes V_d$ we have $\rk^G_\alpha(v)\leq 1$. Moreover, $\rk^G_{\alpha}(v)=1$ if and only if $v$ 
is semi-stable with respect to the group $H=\SL(V_1)\times \SL(V_2)\times \cdots \times \SL(V_d)$.
\end{proposition}
\begin{proof}
The inequality $\rk_\alpha^G(v)\leq 1$ is obvious.
Suppose that $v\in V$ is not $H$-semi-stable. Then there exists a $1$-parameter subgroup $\lambda(t)=(\lambda_1(t),\dots,\lambda_d(t)):\Gm\to H$
with $\lim_{t\to 0}\lambda(t)\cdot v=0$.  
We can choose $c_1,c_2,\dots,c_d$ such that $\lambda'(t)=(t^{c_1}\lambda_1(t),\dots,t^{c_d}\lambda_d(t))\in G(K[t])$.
Note that $\det(t^{c_i}\lambda_i(t))=\det(t^{c_i}I_{n_i})\det(\lambda_i(t))=t^{c_in_i}$.
Now we have have $\vv(\lambda'(t)\cdot v)=s+c_1+c_2+\cdots+c_d$
and
$$
\mu(\lambda'(t),v)=\frac{\sum_{i=1}^d \frac{1}{n_i}\vv(\det (t^{c_i}\lambda_i(t)) )}{\vv(\lambda'(t)\cdot v)}=\frac{\sum_{i=1}^d c_i}{s+\sum_{i=1}^d c_i} <1.
$$
This proves that $\rk^G_\alpha(v)<1$.
 
Conversely, suppose that $\rk^G_\alpha(v)<1$. Choose a polynomial $1$-parameter subgroup of $G$ such that $\vv(\lambda(t)\cdot v)=s>0$ and
 $\mu_{\alpha}(\lambda(t),v)<1$.
Let $c_i=\vv(\det \lambda_i(t))$. Then we have $\mu_\alpha(\lambda(t),v)=\sum_{i=1}^d \frac{c_i}{n_i}<s$.
After replacing $t$ by $t^{k}$ for some positive integer $k$ we may assume that $\frac{c_i}{n_i}\in \Z$ for all $i$.
Let $\lambda'(t)=(t^{-c_1/n_1}\lambda_1(t),t^{-c_2/n_2}\lambda_2(t),\dots,t^{-c_d/n_d}\lambda_d(t))$. Then
$\lambda'(t)$ is a $1$-parameter subgroup of $H$
and $\vv(\lambda'(t)\cdot v)=s-\sum_{i=1}^d \frac{c_i}{n_i}>0$, so $\lim_{t\to 0}\lambda'(t)\cdot v=0$.
This shows that $v$ is $H$-unstable.
\end{proof}
\subsection{The $G$-stable rank and the non-commutative rank}
The non-commutative rank is defined as the rank
of $A(t)=t_1A_1+t_2A_2+\cdots+t_mA_m$ where $t_1,t_2,\dots,t_m$
are variables in the free skew field $R=K\!\!<\!\!\!\!\!(\,\,t_1,t_2,\dots,t_m\!>\!\!\!\!)\ $ and $A(t)$ is viewed as a $p\times q$ matrix with entries in $R$ (see~\cite{FortinReutenauer04,Cohn95} for more on free skew fields).
We will use the following equivalent definition (see~\cite{FortinReutenauer04}):
\begin{definition}
Suppose that $A_1,A_2,\dots,A_m$ are $p\times q$ matrices.
Then the the non-commutative rank $\ncrk(A)$ of $A=(A_1,\dots,A_m)$
is equal to the maximal value of
$$
q+\dim \sum_{i=1}^m A_i(W)-\dim W
$$
over all subspaces $W\subseteq K^q$.
\end{definition}

It was shown in \cite{IQS17} that the non-commutative rank of $A$ is also equal to maximum of
$$
\frac{\rk(\sum_{i=1}^m T_i\boxtimes A_i)}{d}
$$
where $d$ is a positive integer, $T_1,T_2,\dots,T_m$ are $d\times d$ matrices, and $\boxtimes$ is the Kronecker product of two matrices (so $T_i\boxtimes A_i$ is a $dp\times dq$-matrix).

The non-commutative rank  relates to stability. If $A$ is an $m$-tuple of $n\times n$ matrices (i.e., $p=q=n$) then $\ncrk(A)=n$ if and only if $A$ is semi-stable with respect to the simultaneous left-right action of $\SL_n\times \SL_n$ on $m$-tuples of matrices (see~\cite{IQS17}).

We can relate the non-commutative and $G$-stable rank as follows. First, we will view the $m$-tuple $A=(A_1,A_2,\dots,A_m)$ as a tensor. Using a linear isomorphism $K^p\otimes K^{q}\cong K^{p\times q}$, we can view $A_1,A_2,\dots,A_m$ as tensors in $K^p\otimes K^q$.
The $m$-tuple $A=(A_1,A_2,\dots,A_m)$ corresponds to a tensor
$T_A=\sum_{i=1}^m A_i\otimes [i]\in K^p\otimes K^q\otimes K^m$.
\begin{lemma}
The non-commutative rank is the smallest value of $r+s$
for which there exist linearly independent vectors $v_1,\dots,v_r\in K^p$ and linearly independent vectors $w_1,\dots,w_s\in K^q$ with
\begin{equation}\label{eq:TA}
T_A\in \sum_{i=1}^r v_i\otimes K^q\otimes K^m+\sum_{j=1}^s K^p\otimes w_j\otimes K^m. 
\end{equation}
\end{lemma}
\begin{proof}
If (\ref{eq:TA}) holds, then take $W$ to be the $(q-s)$-dimensional space
perpendicular to the vectors $w_1,w_2,\dots,w_s$. The space $A_i(W)$
is contained in the span of $v_1,v_2\dots,v_r$.
So the non-commutative rank is at most $q+r-(q-s)=r+s$.

We show that $r+s$ can be equal to $\ncrk(A).$
Suppose that $k=\ncrk(A)$.
For some $s$ there exists an 
 subspace $V\subseteq K^p$ with $k=q+\dim V-\dim W$,
 where $V=\sum_{i=1}^m A_i(W)$.
 Choose a basis $w_1,w_2,\dots,w_s$ of the space orthogonal
 to $W$. Then we have $s=q-\dim W$.
 Also choose a basis $v_1,v_2,\dots,v_r$ of $V$.
 Now (\ref{eq:TA}) holds and  $r+s=q-\dim W+\dim V=k$.
\end{proof}
The following proposition shows that the non-commutative rank can be seen as a special case of the $G$-stable rank.
\begin{proposition}
For $\alpha=(1,1,\ell)$ and $\ell\geq\min\{p,q\}$ we have $\ncrk(A)=\rk^G_\alpha(T_A)$.
\end{proposition}
\begin{proof}
Let $k=\ncrk(A)$. Then we have
$$
T_A\in \sum_{i=1}^r v_i\otimes K^q\otimes K^m+\sum_{j=1}^s K^p\otimes w_j\otimes K^m.
$$
for some $r$ and $s$ with $r+s=k$ and vectors $v_1,\dots,v_r,w_1,\dots,w_s$.
We extend $v_1,\dots,v_r$ to a basis $v_1,\dots,v_p$ and extend $w_1,\dots,w_s$ to a basis $w_1,\dots,w_q$.
We define a $1$-parameter subgroup $\lambda(t)=(\lambda_1(t),\lambda_2(t),\lambda_3(t))$ in $G=\GL_p\times \GL_q\times \GL_m$ by $\lambda_1(t)\cdot v_i=t v_i$ for $i=1,2,\dots,r$, $\lambda_1(t)\cdot v_i=v_i$ for $i=r+1,r+2,\dots,p$,
$\lambda_2(t)\cdot w_j=tw_j$ for $j=1,2,\dots,s$, $\lambda_2(t)\cdot w_j=w_j$ for $j=s+1,s+2,\dots,q$ and $\lambda_3(t)$ is just the identity. Then we have $\vv(\lambda(t)\cdot T_A)=1$, $\det(\lambda_1(t))=t^r$, $\det(\lambda_2(t))=t^s$, $\det(\lambda_3(t))=1$ and 
$$\rk_{\alpha}^G(T_A)\leq \mu_{\alpha}(\lambda(t),T_A)=\frac{1\cdot r+1\cdot s+\ell\cdot 0}{1}=k=\ncrk(A).$$

On the other hand, let $h=\rk_\alpha^G(T_A)$ and 
suppose that $\lambda(t)\in G$ is a $1$-parameter subgroup with 
$\mu_{\alpha}(\lambda(t),T_A)=h$. If $h=\min\{p,q\}$ then clearly $\ncrk(A)\leq h$,
so we assume that $h<\min\{p,q\}$.
Suppose $\ell\geq p$ (the case $\ell\geq q$ will go similarly). 
If $\det(\lambda_3(t))=t^e$ then we can define another $1$-parameter subgroup $\rho(t)=(\rho_1(t),\rho_2(t),\rho_3(t))$
by $\rho_1(t)=t^e \lambda_1(t)$, $\rho_2(t)=\lambda_2(t)$ and $\rho_3(t)=I$. Then $\vv(\rho(t)\cdot T_A)\geq \vv(\lambda(t)\cdot T_A)$, and we get
\begin{multline*}
\mu_\alpha(\rho(t),T_A)=\frac{\vv(\det \rho_1(t))+\vv(\det \rho_2(t))+\ell \vv(\det \rho_3(t))}{\vv(\rho(t)\cdot T_A)}\leq\\ \leq \frac{pe+\vv(\det \lambda_1(t))+\vv \det(\lambda_2(t))}{\vv(\lambda(t)\cdot T_A)}\leq \\ \leq 
\frac{\vv(\det \lambda_1(t))+\vv(\det \lambda_2(t))+\ell \vv(\det \lambda_3(t))}{\vv(\lambda(t)\cdot T_A)}=\mu_{\alpha}(\lambda(t),T_A)
\end{multline*}
because $\ell\geq p$ and $\vv(\det \lambda_3(t))=e$.
We can replace $\lambda(t)$ by $\rho(t)$ and without loss of generality we may assume that $\lambda_3(t)=I$.

Let $d:=\vv(\lambda(t)\cdot T_A)$. After base changes, we have
$$
\lambda(t)=\begin{pmatrix}
t^{x(1)} & & \\
& \ddots & \\
& & t^{x(p)}
\end{pmatrix}\mbox{ and }
\rho(t)=\begin{pmatrix}
t^{y(1)} & & \\
& \ddots & \\
& & t^{y(q)}
\end{pmatrix}
$$
From 
$$
\frac{\sum_{i=1}^{h+1} (x(i)+y(h+2-i))}{d}\leq \frac{\sum_{i=1}^p x(i)+\sum_{j=1}^q y(j)}{d}=\mu_{\alpha}(\lambda(t),T_A)=h$$ 
follows that $x(r+1)+y(s+1)\leq \frac{hd}{k+1}<hd$ for some $r,s$ with $r+s=h$.
If a basis vector $[i,j,k]=[i]\otimes [j]\otimes [k]$ appears in $T_A$ then $x(i)+y(j)\geq dk$
and therefore $i\leq r$ or $j\leq s$.
 This means that 
$$
T_A\in \sum_{i=1}^r [i]\otimes K^q\otimes K^m+\sum_{j=1}^s K^p\otimes [j]\otimes K^m
$$
and $\ncrk(T_A)\leq r+s=h=\rk_\alpha^G(T_A)$.

\end{proof}

\subsection{Semi-continuity of the $G$-stable rank}
We will show that the $G$-stable rank is semi-continuous, which means that for every $r$, the set  of all tensors with $G$-stable rank $\leq r$
is Zariski closed.

Let us for the moment fix a $1$-parameter subgroup $\lambda(t)$ of $G$.
We can choose bases in the vector spaces $V_i$ for $i=1,2,\dots,d$
such that the matrix of $\lambda_i(t)$ is diagonal, with diagonal entries $t^{x(i,1)},t^{x(i,2)},\dots,t^{x(i,n_i)}$
where $x(i,1)\geq x(i,2)\geq \cdots\geq x(i,n_i)\geq 0$.
Define
$$
Z=\{v\in V\mid  \mu_{\alpha}(\lambda(t),v)<r\}.
$$
The space $Z$ is spanned by all basis vectors
$[i_1,i_2,\dots,i_d]\in V$ with
$$ \sum_{i=1}^d \alpha_i\sum_{j=1}^{n_i}x(i,j)<r(x(1,i_1)+x(2,i_2)+\cdots+x(d,i_d)).$$
Let $B=B_{n_1}\times B_{n_2}\times \cdots\times B_{n_d}\subseteq G$ where $B_{k}\subseteq \GL_k$ is the Borel group of upper triangular invertible matrices.
If $[i_1,i_2,\dots,i_d]$ lies in $Z$, and $j_k\leq i_k$ for all $k$, then $[j_1,j_2,\dots,j_d]$ lies in $Z$.
This implies that $Z$ is stable under the action of $B$.
\begin{lemma}
The set $G\cdot Z=\bigcup_{g\in G} g\cdot Z$ is Zariski closed.
\end{lemma}
\begin{proof}
Consider the Zariski closed subset $S\subseteq G/B\times V$ defined by
$$
S=\{(gB,v)\mid g^{-1}\cdot v\in Z\}
$$
and let $\pi:G/B\times V\to V$ be the projection onto $V$. 
The flag variety $G/B$ is projective, so $\pi$ is a projective morphism which maps closed
sets to closed sets. In particular, $G\cdot Z=\pi(S)$ is Zariski closed.
\end{proof}
\begin{theorem}\label{theo:ZariskiClosed}
For any weight $\alpha\in \R_{>0}^d$ and $r\in \R$ the sets $X^\circ(\rk^G_{\alpha},r)=\{v\in V\mid \rk^G_{\alpha}(v)<r\}$ and
$X(\rk^G_\alpha,r)=\{v\in V\mid \rk^G_{\alpha}(v)\leq r\}$ are finite unions of sets of the form $G\cdot Z$
where $Z$ is a Borel-fixed subspace. In particular, these sets are Zariski closed.
\end{theorem}
\begin{proof}
If $\rk^G_\alpha(v)<r$, then there exists a $1$-parameter subgroup $\lambda(t)$ of $G$ such that  $\mu_{\alpha}(\lambda(t),v)<r$.
If $Z=\{w\in V\mid \mu_{\alpha}(\lambda(t),w)<r\}$ then  $X^\circ(\rk^G_\alpha,r)$ contains $Z$ and $G\cdot Z$.
Since there are only finite many Borel stable subspaces of $V$, we see that $X^\circ(\rk^G_\alpha,r)$ must
be a finite union $G\cdot Z_1\cup G\cdot Z_2\cup \cdots\cup G\cdot Z_s$
where $Z_1,Z_2,\dots,Z_s$ are Borel stable subspaces. Since each $G\cdot Z_i$ is closed, $X^\circ(\rk^G_\alpha,r)$
is closed. Because  there are only finitely many Borel stable subspaces, there are only finitely many possibilities
for $X^\circ(\rk^G_\alpha,s)$ where $s\in \R_{>0}$.
There exists an $\varepsilon>0$ such that $X^\circ(\rk^G_\alpha,s)$ is the same for all $s\in (r,r+\varepsilon]$.
We have $X(\rk^G_\alpha,r)=\bigcap_{r<s\leq r+\varepsilon} X^\circ(\rk^G_\alpha,s)=X^\circ(\rk^G_\alpha,r+\varepsilon)$.
\end{proof}

\section{Results on the $G$-stable rank}
\subsection{Easy observations and a technical lemma}

\begin{lemma}\label{lem:Granklb}
If $v\neq 0$, then we have $\rk^G_\alpha(v)\geq \min\{\alpha_1,\alpha_2,\dots,\alpha_d\}>0$. In particular, $\rk^G(v)\geq 1$.
\end{lemma}
\begin{proof}
Choose $g(t)\in G(\PS)$ with $\mu_\alpha(g(t),v)=\rk^G_\alpha(v)$.
From $v\neq 0$ follows that $g(t)\cdot v\neq 0$, say $\vv(g(t)\cdot v)=s>0$. Then we get $\sum_{i=1}^d \vv( g_i(t))\geq s$
and 
$$
\frac{\sum_{i=1}^d \alpha_i\vv(g_i(t)) }{\vv(g(t)\cdot v)}\geq \min\{\alpha_1,\dots,\alpha_d\}\frac{\sum_{i=1}^s \vv(g_i(t))}{s}\geq \min\{\alpha_1,\dots,\alpha_d\}.
$$
It follows that
$\rk^G_{\alpha}(v)\geq \min\{\alpha_1,\dots,\alpha_d\}>0$.
\end{proof}
Suppose that $v=u\otimes w$ is nonzero with $u\in V_1$ and $w\in V_2\otimes \cdots \otimes V_d$.
We choose bases in $V_1,\dots,V_d$ such that $u$ is the first basis vector in $V_1$.
We can choose a one parameter subgroup $\lambda(t)$ with 
$$
\lambda_1(t)=\begin{pmatrix}
t & & & \\
& 1 & & \\
& & \ddots & \\
& & & 1
\end{pmatrix}
$$
and $\lambda_k(t)=1_{n_k}$ for $k=2,3,\dots,d$. Then we have $\lambda(t)\cdot v=t v$ and $\mu_{\alpha}(A(t),v)=\alpha_1$.
This shows that $\rk^G_\alpha(v)\leq \alpha_1$. From Lemma~\ref{lem:Granklb} follows that 
$\rk^G_\alpha(v)\leq \alpha_1$. If $v$ has slice rank $1$ concentrated in the $i$-th slice, then $\rk^G_\alpha(v)\leq \alpha_i\leq \max\{\alpha_1,\alpha_2,\dots,\alpha_d\}$.
\begin{corollary}
If $v$ has slice rank $1$, then $\rk^G(v)=1$.
\end{corollary}
\begin{proof}
If $v$ has slice rank $1$, then $\rk^G(v)=\rk^G_{(1,\dots,1)}(v)\leq\max\{1,\dots,1\}=1$ and $\rk^G(v)\geq 1$ by Lemma~\ref{lem:Granklb}.
\end{proof}
\begin{corollary}
If $v$ has rank $1$ then $\rk^G_\alpha(v)=\min\{\alpha_1,\dots,\alpha_d\}$.
\end{corollary}
\begin{proof}
If $v$ has rank $1$ then $\rk^G_\alpha\leq \alpha_i$ for every $i$ and
$\rk^G_\alpha\geq \min\{\alpha_1,\dots,\alpha_d\}$ by Lemma~\ref{lem:Granklb}.
\end{proof}

\begin{proposition}\label{prop:horizontal}
Suppose that $v\in V_1\otimes V_2\otimes \cdots \otimes V_d$ and 
$w\in W_1\otimes W_2\otimes \cdots \otimes W_e$
and $v\otimes w\in V_1\otimes \cdots\otimes V_d\otimes W_1\otimes \cdots\otimes W_e$ is the horizontal tensor product.
We have $\rk_{\alpha,\beta}^G(v\otimes w)=\min\{\rk_{\alpha}^G(v),\rk_\beta^G(w)\}.$
\end{proposition}
\begin{proof}
Let $G=\GL(V_1)\times \cdots\times \GL(V_d)$ and $H=\GL(W_1)\times \cdots \times \GL(W_e)$.
There exists $g(t)\in G(K[[t]])$ with $\mu_\alpha(g(t),v)=\rk_\alpha(v)$. 
For $(g(t),1)\in (G\times H)(K[[t]])$ we get
$\mu_{\alpha,\beta}((g(t),h(t)),v\otimes w)=\rk_{\alpha}(v)$.
This proves that $\rk_{\alpha,\beta}^G(v\otimes w)\leq \rk_{\alpha}(v)$. Similarly, we have $\rk_{\alpha,\beta}^G(v\otimes w)\leq \rk^G_\beta(w)$, so we get  $\rk_{\alpha,\beta}^G(v\otimes w)\leq \min\{\rk_\alpha^G(v),\rk_\beta^G(w)\}$.

Conversely, suppose that $(g(t),h(t))\in G\times H(K[[t]])$
satisfies $\mu_{\alpha,\beta}((g(t),h(t)),v\otimes w)=\rk_{\alpha,\beta}^G(v\otimes w)$.
Using that
$$
\vv((g(t),h(t))\cdot (v\otimes w))=\vv((g(t)\cdot v)\otimes (h(t)\cdot w))=\vv(g(t)\cdot v)+\vv(h(t)\cdot w)
$$
we get
\begin{multline*}
\mu_{\alpha,\beta}(v\otimes w)=\frac{\sum_{i=1}^d \vv(\det g_i(t))+\sum_{j=1}^e \vv(\det h_j(t))}{\vv(g(t)\cdot v)+ \vv(h(t)\cdot w)}=\\
\geq \min\left\{
\frac{\sum_{i=1}^d \vv(\det g_i(t))}{ \vv(g(t)\cdot v)},
\frac{\sum_{j=1}^e \vv(\det h_j(t))}{\vv(h(t)\cdot w)}\right\}=
\min\{\rk_\alpha^G(v),\rk_\beta^G(w)\}.
\end{multline*}

\end{proof}

We will need the following technical lemma to prove Proposition~\ref{prop:triangle}.
\begin{lemma}\label{lem:gt}
If $g(t),h(t)\in \GL_n(K[[t]])$ then there exists $u(t),g'(t),h'(t)\in \GL_n(K[[t]])$ such that
$u(t)=g'(t)h(t)=h'(t)g(t)$ and $\vv(\det u(t))\leq \vv(\det g(t))+\vv(\det h(t))$.
\end{lemma}
\begin{proof}
We have
$$\vv(\det g(t))=\dim_K \frac{K[[t]]^n}{g(t)K[[t]]^n}.$$
The $\PS$-module $g(t)K[[t]]^n\cap h(t)K[[t]]^n$ is a submodule of the free module $K[[t]]^n$,
so it is also free of rank $\leq n$. So there exists a matrix $u(t)$ such that $g(t)K[[t]]^n\cap h(t)K[[t]]^n=u(t)K[[t]]^n$.
From $u(t)K[[t]]^n\subseteq g(t)K[[t]]^n$ follows that there exists a matrix $h'(t)$ such that $u(t)=h'(t)g(t)$.
Similarly, we find a matrix $g'(t)$ with $u(t)=g'(t)h(t)$.

We have 
\begin{multline*}
\vv(\det u(t))\leq \dim \frac{K[[t]]^n}{u(t)K[[t]]^n}=\dim \frac{K[[t]]^n}{g(t)K[[t]]^n\cap h(t)K[[t]]^n}=\\=
\dim \frac{K[[t]]^n}{g(t)K[[t]]^n}+\dim \frac{g(t)K[[t]]^n}{g(t)K[[t]]^n\cap h(t)K[[t]]^n}=\\=
\vv(\det g(t))+\dim \frac{g(t)K[[t]]^n+h(t)K[[t]]^n}{h(t)K[[t]]^n}\leq \vv(\det g(t))+\vv(\det h(t)).
\end{multline*}
\end{proof}

\subsection{The triangle inequality for the $G$-stable rank}

\begin{proposition}\label{prop:triangle}
For tensors $v,w\in V$ we have $\rk_{\alpha}^G(v+w)\leq \rk_{\alpha}^G(v)+\rk_{\alpha}^G(w)$.
\end{proposition}
\begin{proof}
Suppose that $g(t),h(t)\in G(K[[t]])$. If we replace $t$ by $t^e$, then $\mu_\alpha(g(t),v)$ does not change.
Without changing $\mu_{\alpha}(g(t),v)$ and $\mu_{\alpha}(h(t),w)$ we may assume that $\vv(g(t)\cdot v)=\vv(h(t)\cdot w)=s>0$.
 Then there exist  $u(t),g'(t),h'(t)\in G(K[[t]])$
such that $u(t)=h'(t)g(t)=g'(t)h(t)$ and $\vv(\det u_i(t))\leq \vv(\det g_i(t) )+\vv(\det h_i(t) )$ for all $i$ by Lemma~\ref{lem:gt}.
We get
\begin{multline*}
\vv(u(t)\cdot (v+w))=\vv(h'(t)g(t)\cdot v+g'(t)h(t)\cdot w)\geq\\ \geq
\min\{\vv(h'(t)g(t)\cdot v),\vv(g'(t)h(t)\cdot w\} \geq \min\{\vv(g(t)\cdot v),\vv(h(t)\cdot w)\}=s
\end{multline*}
and
$$
\sum_{i=1}^d \alpha_i \vv(\det u_i(t))\leq \sum_{i=1}^d \alpha_i \vv(\det g_i(t))+ \sum_{i=1}^d \alpha_i \vv(\det h_i(t))=
s\mu_{\alpha}(g(t),v)+s\mu_{\alpha}(h(t),w).
$$
It follows that
\begin{multline*}
\mu_\alpha(u(t),v+w)=\frac{\sum_{i=1}^d \alpha_i \vv(\det u_i(t))}{\vv(u(t)\cdot (v+w))}\leq\\ \leq \frac{s\mu_{\alpha}(g(t),v)+s\mu_{\alpha}(h(t),w)}{s}=
\mu_{\alpha}(g(t),v)+\mu_{\alpha}(h(t),w).
\end{multline*}
Taking the infimum over all $g(t)$ and $h(t)$ gives
$\rk^G_\alpha(v+w)\leq \rk^G_\alpha(v)+\rk^G_{\alpha}(w)$.
\end{proof}
\begin{corollary}\label{cor:GstableSlice}
For any tensor $v\in V$ we have
$$
\rk^G(v)\leq \slicerk(v).
$$
\end{corollary}
\begin{proof}
By definition, we can write $v=v_1+v_2+\cdots+v_r$ where $r=\slicerk(v)$ and $v_1,v_2,\dots,v_r$ are tensors of slice rank $1$.
Now we have $\rk^G(v)=\rk^G(v_1+\cdots+v_r)\leq \rk^G(v_1)+\cdots+\rk^G(v_r)=1+\cdots+1=r=\slicerk(v)$.
\end{proof}

\subsection{The additive property of the $G$-stable rank}

\begin{proposition}\label{prop:additive}
If $d\geq 2$, the $G$-stable rank is additive: we have $\rk^G_\alpha(v\boxplus w)=\rk^G_\alpha(v)+\rk^G_\alpha(w)$.
\end{proposition}
\begin{proof}
From Proposition~\ref{prop:triangle} follows that $\rk^G_\alpha(v\boxplus w)\leq \rk^G_\alpha(v\boxplus 0)+\rk^G_\alpha(0\boxplus w)\leq \rk^G_\alpha(v)+\rk^G_\alpha(w)$.
Suppose that $g(t)\in G(K[[t]])$ with
$\vv(g(t)\cdot (v\boxplus w))=t^s$ for some $s>0$. Assume that the block form of $g_i(t)$ with respect to the decomposition
$V_i\oplus W_i$ is
$$
g_i(t)=\begin{pmatrix}
a_i(t) & b_i(t)\\
c_i(t) & d_i(t)
\end{pmatrix}.
$$
The $K[[t]]$-module generated by the rows of $a_1(t)$ and $c_1(t)$ is a free submodule of $K[[t]]^{n_1}$
of rank $n_1$,  where $n_1=\dim V_i$.
Using the Smith normal form, there exist invertible matrices in $p(t)\in \GL_{n_1+m_1}(K[[t]])$ and $q(t)\in \GL_{n_1}(K[[t]])$
such that
$$
\begin{pmatrix}
a_1(t)\\
c_1(t)
\end{pmatrix}=p(t)\begin{pmatrix}r(t)\\ 0\end{pmatrix} q(t)
$$
where $r(t)$ is an $n_1\times n_1$ diagonal matrix.
It follows that 
$$p(t)^{-1}g_1(t)=\begin{pmatrix}
r(t) & \star \\
0 & \star\end{pmatrix}
$$
So without loss of generality, we may assume that $c_1(t)=0$.
A similar argument shows that we may assume without loss of generality that $b_2(t)=b_3(t)=\cdots=b_d(t)=0$.
If we project $g(t)\cdot v\boxplus w$ onto $V$, we get $a(t)\cdot v+b(t)\cdot w=a(t)\cdot v$ because $b_2(t)=0$.
This implies that $\vv(a(t)\cdot v)\geq s$ and $\sum_{i=1}^d \alpha_i\vv(\det a_i(t))\geq s\rk^G_{\alpha}(v)$.
Similarly, the projection of $g(t)\cdot v\boxplus w$ onto $W$ is equal to $c(t)\cdot v+d(t)\cdot w=d(t)\cdot w$ because $c_1(t)=0$.
Therefore, we have $\vv(d(t)\cdot w)\geq s$ and $\sum_{i=1}^d\alpha_i \vv(\det d_i(t))\geq s \rk^G_\alpha(w)$.
Since $\det g_i(t)=\det a_i(t)\det d_i(t)$ because of the upper triangular or lower triangular form of $g_i(t)$, we get
$$
\sum_{i=1}^s \alpha_i \vv(\det g_i(t))=\sum_{i=1}^s \alpha_i \vv(\det a_i(t))+\sum_{i=1}^s \alpha_i \vv(\det d_i(t))\geq
s(\rk^G_\alpha(v)+\rk^G_{\alpha}(w)).
$$
This proves that $\rk^G_\alpha(v\boxplus w)\geq \rk^G_\alpha(v)+\rk^G_\alpha(w)$.
\end{proof}

\section{The stable $T$-rank}
\subsection{The $G$-stable rank and the $T$-stable rank}
The $G$-stable $\alpha$-rank of a tensor $v$ is the maximum of $\mu_{\alpha}(\lambda(t),v)$
where $\lambda(t)$ is a $1$-parameter subgroup of $G$ with $\vv(\lambda(t)\cdot v)>0$. A $1$-parameter subgroup is
contained in some maximal torus $T$ (which itself is contained in some Borel subgroup $B$ of $G$).
We can fix a maximal torus $T$ and consider all $1$-parameter subgroups contained in $T$. Choosing a maximal torus of $G$
corresponds to choosing a basis in each vector space $V_i$. 
So let us choose a basis in each $V_i$ so that we can identify $\GL(V_i)$ with $\GL_{n_i}$.
Let $T_k\subseteq \GL_k$ be the subgroup of invertible diagonal $k\times k$ matrices, and $T=T_{n_1}\times T_{n_2}\times \cdots\times T_{n_d}\subseteq G$. Then $T$ is a maximal torus of $G$.
\begin{definition}
We define the $\alpha$-stable $T$-rank $\rk^T_{\alpha}(v)$ as the infimum over all $\mu_{\alpha}(\lambda(t),v)$ where $\lambda(t)\in T(K[t])$ is a $1$-parameter subgroup of $T$ with $\vv(\lambda(t)\cdot v)>0$.
\end{definition}
Since every $1$-parameter subgroup is conjugate to a $1$-parameter subgroup in the maximal torus, we get the following corollary.
\begin{corollary}
We have
$$
\rk^G_{\alpha}(v)=\inf_{g\in G} \rk^T_\alpha(g\cdot v).
$$
\end{corollary}
\subsection{The $T$-stable rank and linear programming}

For a tensor $v=(v_{i_1,i_2,\dots,i_d})\in V=K^{n_1\times n_2\times \cdots \times n_d}$ we define its support by
$$
\supp(v)=\{(i_1,\dots,i_d)\mid v_{i_1,i_2,\dots,i_d}\neq 0\}.
$$
As we will see, $\rk^T_{\alpha}(v)$ only depends on $\supp(v)$ and $\alpha$. For a nonnegative integer $k$, let $\underline{k}=\{1,2,\dots,k\}$.
We will fix a support $S\subseteq \underline{n}_1\times \underline{n}_2\times \cdots\times \underline{n}_d$
and compute the corresponding $\alpha$-stable $T$-rank. 
\begin{definition}
Let $x(i,j)$ with $1\leq i\leq d$ and $1\leq j\leq n_i$ be real variables and $S\subseteq \underline{n}_1\times \cdots\times \underline{n}_d$ be a support.
The linear program $\LP_\alpha(S)$ asks to minimize $\sum_{i=1}^d \alpha_i\sum_{j=1}^{n_i}x(i,j)$
under the constraints:
\begin{enumerate}
\item $x(i,j)\geq 0$ for $i=1,2,\dots,d$ and $1\leq j\leq n_i$;
\item $\sum_{i=1}^{d}x(i,s_i)\geq 1$ for all $s\in S$,
\end{enumerate}
\end{definition}
\begin{theorem}
If $v\in V$ has support $S$, then $\rk^T_\alpha(v)$ is the value of the linear program $\LP_\alpha(S)$.
\end{theorem}
\begin{proof}
Suppose $\lambda(t)=(\lambda_1(t),\dots,\lambda_d(t))\in T(K[t])$ is a $1$-parameter subgroup,
and $\lambda_i(t)$ is diagonal with entries $t^{x(i,1)},t^{x(i,2)},\cdots,t^{x(i,n_i)}$ where $x(i,j)$ is a nonnegative integer for all $i,j$.
Also, assume that $\vv(\lambda(t)\cdot v)=q>0$ where $v$ is a tensor with support $S$.
This means that $\sum_{i=1}^d \alpha_ix(i,s_i)\geq q$ for all $(s_1,s_2,\dots,s_d)\in S$.
We have $\mu_{\alpha}(\lambda(t),v)=\frac{1}{q}(\sum_{i=1}^d\alpha_i\sum_{j=1}^{n_i}x(i,j))$
and $\rk^T_\alpha(v)$ is the infimum of all $\mu_{\alpha}(\lambda(t),v)$.
If we replace $x(i,j)$ by $x(i,j)/q$, then
we have $\sum_{i=1}^d \alpha_ix(i,s_i)\geq 1$ for all $(s_1,\dots,s_d)\in S$
and $\mu_{\alpha}(\lambda(t),v)=\sum_{i=1}^d\alpha_i\sum_{j=1}^{n_i}x(i,j)$. This shows
that $\rk^T_\alpha(v)$ is the infimum of  $\sum_{i=1}^d\alpha_i\sum_{j=1}^{n_i}x(i,j)$ under the constraints
$x(i,j)\geq 0$ for all $i,j$, and $\sum_{i=1}^{d}x(i,s_i)\geq 1$ for all $s\in S$ for all $i,j$.
This is the linear program $\LP_\alpha(S)$, except that the numbers $x(i,j)$ have to be rational. 
However, since the constraints are inequalities with coefficients in $\Q$, there exists an optimal solution over $\Q$.
\end{proof}

\begin{example}\label{example:222tensor}
Consider the tensor 
$$v=[2,1,1]+[1,2,1]+[1,1,2]\in K^{2\times 2\times 2}=K^2\otimes K^2\otimes K^2.
$$ 
with support $S=\{(2,1,1),(1,2,1),(1,1,2)\}$.
We have to solve the following linear program $\LP(S)=\LP_{(1,1,1)}(S)$: minimize $\sum_{i=1}^3\sum_{j=1}^2 x(i,j)$
under the constraints $x(i,j)\geq 0$ for $i=1,2,3$ and $j=1,2$ and
\begin{eqnarray*}
x(1,2)+x(2,1)+x(3,1) & \geq &1\\
x(1,1)+x(2,2)+x(3,1) &\geq & 1\\
x(1,1)+x(2,1)+x(3,2) & \geq & 1
\end{eqnarray*}
 An optimal solution is $x(1,1)=x(2,1)=x(3,1)=\frac{1}{2}$ and $x(1,2)=x(2,2)=x(3,2)=0$.
 So the optimal value is $\rk^T(v)=3\cdot \frac{1}{2}=\frac{3}{2}$. It follows that $\rk^G(v)\leq \rk^T(v)\leq \frac{3}{2}$.
 It is easy to see that $\slicerk(v)>1$ (and thus equal $2$). We will show that $\rk^G(v)=\frac{3}{2}$.
 
Suppose that $\rk^G(v)<\frac{3}{2}$. Then there exists a tensor
  $w\in K^{2\times 2\times 2}$ in the same $G$-orbit as $v$ such that  $\rk^T(w)<\frac{3}{2}$. 
  Let  $S'=\supp(w)\subseteq \underline{2}\times\underline{2}\times \underline{2}$ be the support of $w$.
Also assume that $\{x(i,j)\}$ is an optimal solution for the linear program $\LP(S')$.
 By permuting coordinates, we may assume that $x(i,1)\geq x(i,2)$ for $i=1,2,3$.
 The support $S'$ is not contained in $\{1\}\times \{1,2\}\times \{1,2\}$ because otherwise $w$ and $v$ would have slice rank $1$.
 Therefore, $(2,i,j)\in S'$ for some $i,j$. Because of the ordering of the variables $x(i,j)$, $(2,1,1)\in S'$.
 Similarly, $(1,2,1),(1,1,2)\in S'$. Now $\supp(w)=S'\supseteq S=\supp(v)$, so $\rk^T(w)\geq \rk^T(v)=\frac{3}{2}$.
 Contradiction. 
\end{example}
\subsection{Comparison between the $G$-stable rank and the slice rank}

Besides the slice rank, we will also define a slice rank relative to a maximal torus $T$, or equivalently, relative to bases choices for $V_1,V_2,\dots,V_d$.
\begin{definition}
We say that a tensor $v$ has $T$-slice rank $1$ if $v$ is contained in a space of the form 
$$V_{i,j}=V_1\otimes V_2\otimes \cdots\otimes V_{i-1}\otimes [j]\otimes V_{i+1}\otimes \cdots \otimes V_d.
$$
Now the $T$-slice rank $\slicerk^T(v)$ of an arbitrary tensor $v$ is the smallest nonnegative integer $r$ such that $v$ is a sum of $r$ tensors of $T$-slice rank $1$.
\end{definition}
 The following result is clear from the definition of slice rank:
\begin{corollary}
We have
$$
\slicerk(v)=\min_{g\in G} \slicerk^T(g\cdot v).
$$
\end{corollary}
The $T$-slice rank of $v$ depends only on its support $S=\supp(v)$ and can be expressed in terms of integer solutions of the linear program $\LP(S)$.
\begin{proposition}
The $T$-slice rank $\slicerk^T(v)$ is the smallest possible value of $\sum_{i=1}^d \sum_{j=1}^{n_i} x(i,j)$ where the $x(i,j)$ satisfy the constraints:
\begin{enumerate}
\item $x(i,j)\in \{0,1\}$ for $i=1,2,\dots,d$ and $1\leq j\leq n_i$;
\item $\sum_{i=1}^{d}x(i,s_i)\geq 1$ for all $s\in S$;
\end{enumerate}
\end{proposition}
\begin{proof}
Suppose that $x(i,j)\in \{0,1\}$ for all $i,j$.
Define  
$$V(x)=\sum_{\scriptstyle i,j\atop \scriptstyle x(i,j)=1}V_{i,j}.$$
A vector $[s_1,s_2,\dots,s_d]$ lies in $V(x)$ if and only if $\sum_{i=1}^d x(i,s_i)\geq 1$.
So a tensor $v$ lies in $V(x)$ if and only if $\sum_{i=1}^d x(i,s_i)\geq 1$ for all $s\in \supp(v)$.
 By definition, $\slicerk^T(v)$ is the smallest possible value of $\sum_{i,j}x(i,j)$
such that $v\in V(x)$.
\end{proof}
It is now easy to see that $\rk^T(v)\geq \frac{1}{d}\srk^T(v)$ (and this implies $\rk^G(v)\geq \frac{1}{d}\srk(v)$): If $x(i,j)$ is a solution to the linear program ${\bf LP}(S)$
where $S=\supp(v)$, then we define $x'(i,j)\in \{0,1\}$ such that $x'(i,j)=1$ if $x(i,j)\geq \frac{1}{d}$ and $x'(i,j)=0$ otherwise.
If $s\in S$ then we have $\sum_{i=1}^d x(i,s_i)\geq 1$. It follows that $x(i,s_i)\geq \frac{1}{d}$ for some $i$ and $x'(i,s_i)=1$ for some $i$.
Therefore, $\sum_{i=1}^d x'(i,s_i)\geq 1$. Now $\slicerk^T(v)\leq \sum_{i,j}x'(i,j)\leq \sum_{i,j}dx(i,j)=d\rk^T(v)$. With a more refined argument, we can improve this bound:
\begin{proposition}\label{prop:GstableSlice3}
For $d\geq 2$ we have $\rk^T(v)\geq\frac{2}{d} \slicerk^T(v)$ and therefore $\rk^G(v)\geq \frac{2}{d}\slicerk(v)$.
\end{proposition}
\begin{proof}
Suppose that $x(i,j)$ is an optimal solution to the linear program. Note that $0\leq x(i,j)\leq 1$ for all $i,j$.
 We define functions $f_1,f_2,\dots,f_d:[0,1]\to \R$
by 
$$f_i(\alpha)=|\{j\mid x(i,j)\geq \alpha\}|.
$$
 We have $\int_{0}^1 f_i(\alpha)\,d\alpha=\sum_j x(i,j)$. 
In particular, $\int_0^1 (f_1(\alpha)+\cdots+f_d(\alpha))\,d\alpha=\sum_{i,j}x(i,j)$.
Let $s_i= \frac{2i}{d(d-1)}$ for $i=0,1,2,\dots,d-1$. Note that $s_0+s_1+\cdots+s_{d-1}=1$.
We define a closed piecewise linear curve $\gamma=(\gamma_1,\dots,\gamma_d):[0,d]\to \R^d$ with $\gamma(d)=\gamma(0)=[s_0,s_1,\dots,s_{d-1}]$, $\gamma(1)=[s_1,s_2,\dots,s_{d-1},s_0]$,
\ldots, $\gamma(d-1)=[s_{d-1},s_0,\dots,s_{d-2}]$ such that $\gamma$ is linear on each of the intervals $[i,i+1]$, $i=0,1,\dots,d-1$.
On the intervals $[0,1],[1,2],\dots,[d-1,d]$, $\gamma_i(t)$ goes through the intervals $[s_0,s_1]$,$[s_1,s_2]$,$\dots$, $[s_{d-2},s_{d-1}]$,$[s_{d-1},s_0]$
in some order. So $\frac{1}{d}\int_0^d f_i(\gamma_i(t))\,dt$ is the average of the averages of $f_i$ of each of these $d$ intervals. 
This is equal to the average value of $f_i(t)$ on the interval $[0,s_{d-1}]=[0,\frac{2}{d}]$:
$$
 {\textstyle\frac{1}{d}}\int_0^df_i(\gamma_i(t))\,dt={\textstyle\frac{d}{2}}\int_0^{\frac{2}{d}} f_i(t)\,dt\leq {\textstyle\frac{d}{2}}\int_0^1f_i(t)\,dt
={\textstyle \frac{d}{2}}\sum_{j=1}^{n_i}x(i,j).
$$
It follows that
$$
 {\textstyle\frac{1}{d}}\int_0^d\Big(\sum_{i=1}^df_i(\gamma_i(t))\Big)\,dt\leq {\textstyle \frac{d}{2}}\sum_{i=1}^d \sum_{j=1}^{n_i}x(i,j)={\textstyle\frac{d}{2}}\rk^T(v).
$$
Since the minimal value of $\sum_{i=1}^d f_i(\gamma_i(t))$ is at most the average, there exists a $t\in [0,d]$ such that $\sum_{i=1}^d f_i(\gamma_i(t))\leq \frac{d}{2}\rk^T(v)$.
Now define $x'(i,j)=1$ if $x(i,j)\geq \gamma_i(t)$ and $x'(i,j)=0$ if $x(i,j)<\gamma_i(t)$. If $s=(s_1,s_2,\dots,s_d)\in\supp(v)$,
then $\sum_{i=1}^d x(i,s_i)\geq 1$. Since $\sum_{i=1}^d \gamma_i(t)=1$, we have $x(i,s_i)\geq \gamma_i(t)$ for some $i$
and $\sum_{i=1}^d x'(i,s_i)\geq 1$. We conclude that
$$
\slicerk^T(v)\leq \sum_{i=1}^n\sum_{j=1}^{n_i}x'(i,j)=\sum_{i=1}^d f_i(\gamma_i(t))\leq{\textstyle \frac{d}{2}} \rk^T(v).
$$
Finally, we get
$$
\slicerk(v)=\inf_{g\in G} \slicerk^T(g\cdot v)\leq{\textstyle \frac{d}{2}} \inf_{g\in G} \rk^T(g\cdot v)={\textstyle \frac{d}{2}}\rk^G(v).
$$
\end{proof}

\subsection{The dual program and the $T$-stable rank}

\begin{definition}
For a support set $S$, the dual program $\LP^{\vee}_\alpha(S)$ is to maximize $\sum_{s\in S}y(s)$
under the constraints
\begin{enumerate}
\item $y(s)\geq 0$ for all $s\in S$;
\item for all $i,j$ we have
 $$\sum_{\scriptstyle s\in S\atop \scriptstyle s_i=j} y(s)\leq \alpha_i.
 $$
\end{enumerate}
\end{definition}
If $x$ and $y$ are optimal solutions for $\LP_\alpha(S)$ and $\LP_\alpha^\vee(S)$ respectively, then we have
$$\sum_{s\in S}y(s)=\sum_{i=1}^d\alpha_i\sum_{j=1}^{n_i}x(i,j)=\rk^T_\alpha(v)$$ and 
\begin{enumerate}
\item for all $i, j$, we have 
$$\sum_{\scriptstyle s\in S\atop \scriptstyle s_i=j}y(s)=\alpha_i\mbox{ or }x(i,j)=0;$$
\item for all $s\in S$ we have $\sum_{i=1}^d x(i,s_i)=1$ or $y(s)=0$.
\end{enumerate}

\subsection{The super-multiplicative property of the $T$-stable rank}

If $v\in V=V_1\otimes V_2\otimes \cdots\otimes V_d$ and $w\in W_1\otimes W_2\otimes \cdots\otimes W_d$
then we can consider the ``vertical'' tensor product $v\boxtimes w\in (V_1\otimes W_1)\otimes \cdots (V_d\otimes W_d)$.
\begin{proposition}
We have $\rk^T_{\alpha\beta}(v\boxtimes w)\geq \rk^T_\alpha(v)\rk^T_\beta(w)$, where $\alpha=(\alpha_1,\dots,\alpha_d)$,
$\beta=(\beta_1,\dots,\beta_d)$ and $\alpha\beta=(\alpha_1\beta_1,\dots,\alpha_d\beta_d)$.
\end{proposition}
\begin{proof}
Let $S=\supp(v)$, $S'=\supp(w)$,  $y(s),s\in S$ be an optimal solution for the$\LP^\vee_\alpha(v)$ and
 $y'(s),s\in S'$ be an optimal solution for $\LP^\vee_\beta(w)$.
The tensor $v\boxtimes w$ has support $S\times S'$. 
For the dual program for $v\boxtimes w$  we have to maximize $\sum_{s\in S,s'\in S'}Y(s,s')$
under the constraints $Y(s,s')\geq 0$ for all $s\in S,s'\in S'$ and 
$$\sum_{\scriptstyle s\in S,s'\in S'\atop \scriptstyle s_i=j,s'_i=j'}Y(s,s')\leq \alpha_{j}\beta_{j'}$$
for all $i,j,j'$.
One solution for this linear program is $Y(s,s')=y(s)y'(s')$. We get
$$
\rk_{\alpha\beta}^T(v\boxtimes w)\geq \sum_{s\in S}\sum_{s'\in S'} Y(s,s')= \sum_{s\in S}y(s)\sum_{s'\in S'} y(s')=\rk^T_\alpha(v)\rk^T_\beta(w).
$$

\end{proof}

\section{$G$-stable rank over $\C$}
\subsection{Kempf-Ness theory}
We recall some of the main results from Kempf-Ness theory \cite{KempfNess79,Woodward10}. Suppose that $G$ is an complex reductive algebraic group with a maximal compact subgroup $C$
and $V$ is a representation of $G$. We fix a Hermitian inner product $\langle\cdot,\cdot\rangle$ on $V$ that is invariant under $C$, i.e., $\langle g\cdot v,g\cdot w\rangle=\langle v,w\rangle$ for all $v,w\in V$ and $g\in C$.
 Let ${\mathfrak c}$ and ${\mathfrak g}$ be the Lie algebras of $C$ and $G$ respectively, and let ${\mathfrak c}^\star$ be the dual space of ${\mathfrak c}$. We have ${\mathfrak g}={\mathfrak c}\oplus i {\mathfrak c}$.
For $v\in V$, we define a morphism
$\psi_v:G\to \R$ by $g\mapsto \|g\cdot v\|^2=\langle g\cdot v,g\cdot v\rangle$. The differential
$(d\psi_v)_I:\mathfrak g\to \R$  of $\psi_v$ at the identity $I\in G$ is given by 
$$(d\psi_v)_I:\xi\mapsto \langle \xi v,v\rangle+\langle v,\xi v\rangle\in \R$$
Because $\|g\cdot v\|^2$ is constant on $C$, $(d\psi_v)_I$ vanishes on ${\mathfrak c}$.  So $\langle v,\xi v\rangle=-\langle \xi v,v\rangle$ for $\xi\in {\mathfrak c}$.
If $\xi\in {\mathfrak c}$ then we have
$(d\psi_v)_I(i\xi)=\langle i\xi v,v\rangle+\langle v,i\xi v\rangle=i\langle \xi v,v\rangle-i\langle v,\xi v\rangle=2i\langle \xi v,v\rangle$. For the following result, see \cite[Corollary 5.2.5.]{Woodward10}.
\begin{theorem}[Kempf-Ness]
\label{theo:KempfNess}
An orbit $G\cdot v$ is closed if and only there exists $w\in G\cdot v$ with $(d\psi_w)_I=0$.
\end{theorem}
Let $V=V_1\otimes V_2\otimes \cdots \otimes V_d$ with $V_i=\C^{n_i}$. For $v\in V$, let $\Phi_i(v)\in (V_1\otimes \cdots\otimes \widehat{V_i}\otimes \cdots \otimes V_d)^\star \to V_i$ be the $i$-th flattening of $v$.

\subsection{A formula for the $G$-stable rank over $\C$}
We will use Kempf-Ness theory  to prove the following theorem:
\begin{theorem}\label{theo:Cstablerank}
For $\alpha\in \R_{>0}$ we have
$$
\rk_{\alpha}^G(v)=\sup_{g\in G} \min_i \frac{\alpha_i\|g\cdot v\|^2}{\|\Phi_i(g\cdot v)\|^2_\sigma}
$$
\end{theorem}
For the proof of the theorem, we need the following lemma:
\begin{lemma}\label{lem:critical}
Suppose that $\beta\in \Q_{>0}^d$, $r=\frac{p}{q}$ with $p,q$ positive integers, $q\beta\in \Z^d$ and
$v\in V=V_1\otimes V_2\otimes \cdots\otimes V_d$. As in Proposition~\ref{prop:stable}, let
$$
\textstyle W=\left(V^{\otimes p}\otimes \det^{-q\beta} \right)\oplus V_1^{n_1}\oplus V_2^{n_2}\oplus \cdots\oplus V_d^{n_d}.
$$
and $w=(v^{\otimes p}\otimes 1,u_1,\dots,u_d)$. Define $\psi_w:G\to W$ by $\psi_w(g)=g\cdot w$.
Then we have  $(d\psi_w)_I=0$ if and only if
$$
p\|v\|^{2p-2}\Phi_i(v)\Phi_i^\star(v)-q\beta_i\|v\|^{2p}I_{n_i}+u_iu_i^\star=0
$$
for all $i$.
\end{lemma}
\begin{proof}
The Hermitian scalar products on $V_1,V_2,\dots,V_d$ induce Hermitian scalar products on $V_1^{n_1},\dots,V_d^{n_d}$, $V$, $V^{\otimes p}$,
$V^{\otimes p}\otimes \det^{-q\beta}$ and $W$ in a natural way. We have
$$
\textstyle\| w\|^2=\|v\|^{2p}+\sum_{i=1}^d \|u_i\|^2
$$
and
$$
\psi_w(g)=\|g\cdot w\|^2=\|g\cdot v\|^{2p}\textstyle \det^{-2q\beta}(g)+\sum_{i=1}^d \|g_iu_i\|^2.
$$
The Lie algebra of $G$ can be identified with 
$$
{\mathfrak g}=\End(V_1)\oplus \End(V_2)\oplus \cdots\oplus \End(V_d).
$$
The Lie algebra ${\mathfrak c}$ consists of all $d$-tuples $(\xi_1,\dots,\xi_d)$ of skew-Hermitian matrices,
and $i{\mathfrak c}$ consists of $d$-tuples of Hermitian matrices.
We compute the differential $(d\psi_w)_I$. Note that $\GL(V_i)$ acts on the $i$-th mode. If we view $v$ as the flattened tensor
$\Phi_i(v)$, then $g_i$ acts just by left multiplication: $\Phi_i(g_i \cdot v)=g_i\Phi_i(v)$. Let $\Trace(\cdot)$ denote the trace. 
The differential of $g_i\mapsto\|g_i\cdot v\|^2=\Trace(g_i\Phi_i(v)\Phi_i^\star(v)g_i^\star)$ at the
identity is given by $\xi_i\in \End(V_i)\mapsto \Trace(\xi_i\Phi_i(v)\Phi_i^\star(v))+\Trace(\Phi_i(v)\Phi_i^\star(v)\xi_i^\star)$.
If we restrict to Hermitian $\xi_i$, then this  is equal to $2\Trace(\xi_i\Phi_i(v)\Phi_i^\star(v))$.
The differential of $\|g\cdot v\|^2$ restricted to $i{\mathfrak c}\subseteq {\mathfrak g}$ is $(\xi_1,\dots,\xi_d)\mapsto 2\sum_{i=1}^d\Trace(\xi_i\Phi_i(v)\Phi_i^\star(v))$. The differential of $g_i\mapsto \det(g_i)$ at the identity is $\xi_i\mapsto \Trace(\xi_i)$.
Combining these results with the product rule of differentation, we get for $\xi\in i{\mathfrak c}$ that
\begin{multline*}
(d\phi_w)_I(\xi)=\sum_{i=1}^d\Big(2p\|v\|^{2p-2}\Trace(\xi_i\Phi_i(v) \Phi^\star_i(v))-2q\beta_iq\|v\|^{2p}\Trace(\xi_i)+2\Trace(\xi_iu_iu_i^\star)\Big)=\\=
\sum_{i=1}^d\langle \xi_i,\|v\|^{2p-2}\Phi_i(v)\Phi_i^\star(v)-2q\beta_i\|v\|^{2p}I_{n_i}+2u_iu_i^\star\rangle
\end{multline*}
We have $(d\phi_w)_I=0$ if and only if
$$
2p\|v\|^{2p-2}\Phi_i(v)\Phi_i^\star(v)-2q\beta_i\|v\|^{2p}I_{n_i}+2u_iu_i^\star=0
$$
for all $i$.
\end{proof}

\begin{proof}[Proof of Theorem~\ref{theo:Cstablerank}]
Let us define
$$
f_\alpha(v)=\sup_{g\in G} \min_i \frac{\alpha_i\|g\cdot v\|^2}{\|\Phi_i(g\cdot v)\|^2_\sigma}
$$
Suppose that $r\in \Q$ and $f_\alpha(v)\leq r$.  Assume that $\beta\in \Q_{>0}^d$ with $\beta_i>\alpha_i$ for all $i$.
We can write $r=p/q$ such that $p,q\in \Z$ are positive and $q\beta_i\in \Z$ for all $i$. From $f_\alpha(v)\leq r$ follows that
$$
\alpha_i\|g\cdot v\|^2I_{n_i}-r\Phi_i(g\cdot v)\Phi_i^\star(g\cdot v)
$$
is nonnegative definite for all $i$. This implies that
$$
\beta_i\|g\cdot v\|^2I_{n_i}-r\Phi_i(g\cdot v)\Phi_i^\star(g\cdot v)
$$
is positive definite for all $i$. Multiplying with $p\|g\cdot v\|^{2p-2}$ we get that
$$
p\beta_i\|g\cdot v\|^{2p}I_{n_i}-q\|g\cdot v\|^{2p-2}\Phi_i(g\cdot v)\Phi_i^\star(g\cdot v)
$$
is positive definite and equal to $u_iu_i^\star$ for some $u_i\in V_i^{n_i}$. This shows that $(d\psi_{g\cdot w})_I=0$. By Theorem~\ref{theo:KempfNess},
the $G$-orbit of $w$ is closed. By Proposition~\ref{prop:stable}, we have $\rk^G_{\beta}(v)\geq r$. Because this is true for every rational $\beta>\alpha$,
we get $\rk^G_{\alpha}(v)\geq r$. Since this is true for any $r\in \Q$ with $r\geq f_\alpha(v)$, we can conclude that $\rk^G_\alpha(v)\geq f_\alpha(v)$.

 Suppose that $\beta\in \Q_{>0}^d$ and $\beta_i<\alpha_i$ for all $i$. Let $r=\rk^G_{\beta}(v)<\rk^G_\alpha(v)$. 
 We can write $r=\frac{p}{q}$ such that $p$, $q$ are positive integers, and $q\beta\in \Z^d$.
 We can choose an invertible $u_i\in V_i^{n_i}$ for all $i$.
Now 
$$w=(v^{\otimes p}\otimes 1,u_1,u_2,\dots,u_d)\in (V^{\otimes p}\otimes {\textstyle \det^{-q\beta}})\oplus V_1^{n_1}\oplus V_2^{n_2}\oplus \cdots\oplus V_d^{n_d}
$$ is $G$-semi-stable by Proposition~\ref{prop:stable}. So there exists a nonzero $w'\in \overline{G\cdot w}$ with $(d\psi_{w'})_I=0$.
We can write $w'=((v')^{\otimes d}, u_1',\dots,u_d')$. Using Lemma~\ref{lem:critical}, we get
$$
p\|v'\|^{2p-2}\Phi_i(v')\Phi_i^\star(v')-q\beta_i\|v'\|^{2p}I_{n_i}+u_i'(u_i')^\star=0.
$$
So
$$
q\beta_i\|v'\|^{2p}I_{n_i}-p\|v'\|^{2p-2}\Phi_i(v')\Phi_i^\star(v')
$$
is nonnegative definite for all $i$. Therefore, 
$$
q\alpha_i\|v'\|^{2p}I_{n_i}-p\|v'\|^{2p-2}\Phi_i(v')\Phi_i^\star(v')
$$
is positive definite for all $i$.

Since $w'$ lies in $\overline{G\cdot w}$, there exists a $g\in G$ such that
$$
q\alpha_i\|g\cdot v\|^{2p}I_{n_i}-p\|g\cdot v\|^{2p-2}\Phi_i(g\cdot v)\Phi_i^\star(g\cdot v)
$$
is positive definite for all $i$. It follows that
$$
\|\Phi_i(g\cdot v)\|^2_\sigma=\|\Phi_i(g\cdot v)\Phi_i^\star(g\cdot v)\|_\sigma\leq \frac{q\alpha_i\|g\cdot v\|^{2p}}{p\|g\cdot v\|^{2p-2}}=\frac{\alpha_i\|g\cdot v\|^2}{r}
$$
for all $i$ and
$$
\min_{i}\frac{\alpha_i\|g\cdot v\|^2}{\|\Phi_i(g\cdot v)\|_\sigma^2}
\geq r.
$$
This shows that $f_\alpha(v)\geq r=\rk_\beta^G(v)$. Since $\beta\in \Q_{>0}^d$ was arbitrary with $\beta<\alpha$, we obtain $f_{\alpha}(v)\geq \rk_\alpha^G(v)$.
 We conclude that $f_\alpha(v)=\rk^G_{\alpha}(v)$.
\end{proof}

\subsection{The super-multiplicative property of the $G$-stable rank in characteristic 0}

\begin{theorem}\label{theo:supermultiplicative}
If $v\in V_1\otimes V_2\otimes \cdots \otimes V_d$ and $w\in W_1\otimes W_2\otimes \cdots \otimes W_d$ where $V_1,\dots,V_d,W_1,\dots,W_d$ are $\C$-vector spaces 
and $\alpha,\beta\in \R_{>0}^d$, then we have
$$
\rk_{\alpha\beta}^G(v\boxtimes w)\geq \rk_{\alpha}^G(v)\rk_{\beta}^G(w).
$$
\end{theorem}
\begin{proof}
if $g\in \GL(V_1)\times \cdots \times \GL(V_d)$ and $h\in \GL(W_1)\times \cdots \times \GL(W_d)$ then we can consider
 $g\boxtimes h\in \GL(V_1\otimes W_1)\times \cdots \times \GL(V_d\otimes W_d)$. We have
 $$
 \frac{\alpha_i\beta_i\|(g\boxtimes h)\cdot (v\boxtimes w)\|^2}{\|\Phi_i((g\boxtimes h)\cdot (v\boxtimes w))\|_\sigma}=
  \frac{\alpha_i\beta_i\|((g\cdot v)\boxtimes (h\cdot w)\|^2}{\|\Phi_i((g\cdot v)\boxtimes (h\cdot w))\|_\sigma}=
   \frac{\alpha_i\|g\cdot v\|^2\beta_i\| h\cdot w\|^2}{\|\Phi_i(g\cdot v)\|_\sigma\|\Phi_i(h\cdot w)\|_\sigma}
 $$
 Therefore, we get
 $$
 \min_{i} \frac{\alpha_i\beta_i\|(g\boxtimes h)\cdot (v\boxtimes w)\|^2}{\|\Phi_i((g\boxtimes h)\cdot (v\boxtimes w))\|_\sigma}\geq
 \min_{i} \frac{\alpha_i\|g\cdot v\|^2}{\|\Phi_i(g\cdot v)\|_\sigma}\cdot \min_j \frac{\beta_j\|h\cdot w\|^2}{\|\Phi_j(h\cdot w)\|_\sigma}.
 $$
 Taking the supremum over all $g$ and $h$ now gives $\rk^G_{\alpha\beta}(v\boxtimes w)\geq \rk^G_{\alpha}(v)\rk^G_\beta(w)$
 \end{proof}

\section{Application of the $G$-stable rank to the Cap Set Problem}
The Cap Set Problem asks for a largest possible subset $S\subseteq\F_3^n$ without an arithmetic progression. Let $c(n)$ be the largest possible cardinality of such a set. 
 It was recently proved by
Ellenberg and Gijswijt that $c(n)=O(\theta^n)$, where $\theta=\frac{3}{8}(207+33\sqrt{33})^{\frac{1}{3}}<2.756$.  Tao gave an elegant formulation of the proof of this bound using the notion of slice rank. 
Here we will use a similar approach, using the $G$-stable rank instead of the slice rank to get an explicit bound for all $n$
which the same asymptotic behavior. We view $K^3$ as the vector space with basis $[0],[1],[2]$
where we view $0,1,2$ as elements in $\F_3$. More generally, we view $K^{3^n}$ as the vector space with basis $[a]$, $a\in \F_3^n$. Note that $a,b,c$ form an arithmetic progression in $\F_3^n$ if and only if $a+b+c=0$.
Consider the tensor
$$
v_n=\sum_{\scriptstyle (a,b,c)\in \F_3^{n\times 3}\atop \scriptstyle a+b+c=0}[a]\otimes [b]\otimes [c]=\sum_{\scriptstyle (a,b,c)\in \F_3^{n\times 3}\atop \scriptstyle a+b+c=0}[a,b,c]\in K^{3^n}\otimes K^{3^n}\otimes K^{3^n}.
$$
Suppose that $S\subset \F_3^n$ is a set without arithmetic progression. Then we have
$$
w=\sum_{\scriptstyle (a,b,c)\in S^3\atop \scriptstyle a+b+c=0} [a,b,c]\in K^3\otimes K^3\otimes K^3=\sum_{a\in S}[a,a,a]
$$
The tensor $w$ is a projection of $v$ and lies in the orbit closure of $v$. In particular, we have $\rk^G(w)\leq \rk^G(v)$.
Since $w$ is a direct sum of $|S|$ rank 1 tensors, we get $\rk^G(w)\geq |S|$ by Proposition~\ref{prop:additive}. So we have $\rk^G(v)\geq \rk^G(w)\geq |S|$.

We will work over the field $K=\F_3$. For a function $f:\F_3^n\to \F_3$ we define 
$$\langle f\rangle=\sum_{a\in \F_3^n} f(a)[a]\in K^{3^n}.
$$
In particular, we have $\langle 1\rangle=[0]+[1]+[2]$, $\langle x\rangle=[1]+2[2]=[1]-[2]$ and $\langle x^2\rangle=[1]+[2]$.
A basis of $K^{3^n}$ is formed by taking all $\langle p(x)\rangle$ where $p(x)=p(x_1,\dots,x_n)$ is a polynomial of degree $\leq 2$ in each of the variables $x_1,x_2,\dots,x_n$.
With respect to the basis $\langle 1\rangle,\langle x\rangle,\langle x^2\rangle$,
we have $v_n=\langle f\rangle$ where $f:\F_3^n\times \F_3^n\times \F_3^n\to \F_3$ is given by 
$$
f(x,y,z)=\begin{cases}
1 & \mbox{if $x+y+z=0$;}\\
0 &\mbox{otherwise.}
\end{cases}
$$
For $n=1$ we have $v_1=\langle f\rangle$ where $f:\F_3\times \F_3 \times \F_3\to \F_3$ is given by $f(x,y,z)=1-(x+y+z)^2=1-x^2-y^2-z^2+x+y+z$.
So we have
$$
v_1=\langle 1,1,1\rangle -\langle x^2,1,1\rangle-\langle 1,x^2,1\rangle-\langle 1,1,x^2\rangle+\langle 1,x,x\rangle+\langle x,1,x\rangle+\langle x,x,1\rangle.
$$
The support of $S$ with respect to the basis $\langle 1\rangle,\langle x\rangle,\langle x^2\rangle$ is
$$
\{(0,0,0),(2,0,0),(0,2,0),(0,0,2),(0,1,1),(1,0,1),(1,1,0)\}
$$
An optimal solution to the linear program is $x(1,0)=x(2,0)=x(3,0)=\frac{1}{2}$, $x(1,1)=x(2,1)=x(3,1)=\frac{1}{4}$ and $x(1,2)=x(2,2)=x(3,2)=0$,
which gives $\rk^G(v)\geq\rk^T(v)=\sum_{i,j}x(i,j)=\frac{9}{4}=2.25$.
An optimal solution for the dual program is $y(2,0,0)=y(0,2,0)=y(0,0,2)=\frac{1}{4}$  and $y(0,1,1)=y(1,0,1)=y(1,1,0)=\frac{1}{2}$ and $y(0,0,0)=0$.

The support of the tensor $v^{\boxtimes n}=v\boxtimes v\boxtimes \cdots\boxtimes v$ is contained in the set 
$$T_n=\{(\lambda,\mu,\nu)\in (\{0,1,2\}^n)^3\mid
|\lambda|\leq 2n, |\mu|\leq 2n,|\nu|\leq 2n\}.$$
We will give a solution to the linear program $\LP(S^n)$
that we conjecture to be optimal. Whether optimal or not, 
it will give an upper bound for the $G$-stable rank of $v^{\boxtimes n}$.
Suppose that $t_0,t_1,t_2,\dots,t_{2n}\geq 0$ are numbers such that $t_i+t_j+t_k\geq 1$ whenever $i+j+k\leq 2n$. 
If we define $x(i,\lambda)=t_{|\lambda|}$ for all $\lambda\in \{0,1,2\}^n$, and $i=1,2,3$ then we have
$x(1,\lambda)+x(2,\mu)+x(3,\nu)=t_{|\lambda|}+t_{|\mu|}+t_{|\nu|}\geq 1$, so we have a solution to the linear program.
So we get 
$$\rk^G(v)\leq \sum_{i=1}^3\sum_{\lambda} x(i,\lambda)=3\sum_{\lambda} t_{|\lambda|}=3\sum_{i=0}^{2n}f_{n,i}t_i$$
where $f_{n,i}$ is the number of solutions to $a_1+a_2+\cdots+a_n=d$ with $a_1,a_2,\dots,a_n\in \{0,1,2\}$.
So $f_{n,i}$ is the coefficient of $x^i$ in $(1+x+x^2)^n$. To choose the $t$'s optimally, we have to solve a linear program
by minimizing $3\sum_{i=0}^{2n}f_{n,i}t_i$ under the constraints:
\begin{enumerate}
\item $t_i+t_j+t_k\geq 1$ if $i+j+k\leq 2n$; 
\item $t_i\geq 0$ for all $i$.
\end{enumerate}
We get the following optimal solutions for the $t_i$:
$$
\begin{array}{||c|c||c|c|c|c|c|c|c||c|c|c|c||}
\hline\hline
n & & 0 & 1 & 2 & 3 & 4 & 5 & 6 & \mbox{UB} & \mbox{EG'} & \mbox{EG} & \mbox{best cap set} \\ \hline \hline
1 & f_{1,i} & 1 & 1 & 1 & 0 & 0 & 0 & 0  & & & & \\ 
 & t_i & \frac{1}{2} & \frac{1}{4} & 0 & 0 & 0 & 0 & 0& 2\frac{1}{4} & 3 & 3 & 2\\
 \hline
2 & f_{2,i} & 1 & 2 & 3 & 2 & 1 & 0 & 0 & & & & \\
 & t_i &  \frac{3}{5} & \frac{2}{5} & \frac{1}{5} & 0 & 0 & 0 & 0 & 6& 7 & 9 &  4 \\ \hline
3 & f_{3,i} & 1 & 3 & 6 & 7 & 6 & 3 & 1 &  & & &\\
 & t_i & 1 & \frac{2}{3} & \frac{1}{3} & 0 & 0 & 0 & 0& 15 & 18 & 30  &  9 \\ \hline
4 & f_{4,i} & 1 & 4 & 10 & 16 & 19 & 16 & 10 & & & &\\
 & t_i & 1 & \frac{3}{4} & \frac{1}{2} & \frac{1}{4} & 0 & 0 & 0 & 39 & 45 & 45 &  20  \\ \hline
5 & f_{5,i} & 1 & 5 & 15 & 30 & 45 & 51 & 45 & & & &\\
 & t_i & 1 & \frac{4}{5} & \frac{3}{5} & \frac{2}{5} & \frac{1}{5} & 0 &  0 & 105 &  123  & 153 & 45\\ \hline
6 & f_{6,i} & 1 & 6 & 21 & 50  & 90 & 126 & 141& & & &\\
 & t_i &  1 & 1 & 1 & \frac{2}{3} & \frac{1}{3} &  0& 0 & 274 &  324 & 504  & 112\\ \hline\hline
\end{array}
$$
In the table, the column UB gives the value of
$3\sum_{i=0}^{2n}f_{n,i}t_i$ which is an upper bound for 
the $G$-stable rank and the cardinality of a cap set
in $\F_3^n$. 
The column labeled ``best cap set'' gives the cardinality of the largest known cap set in $\F_3^n$. The column EG 
gives the Ellenberg--Gijswijt upper bound, which is $3\sum_{i=0}^{\lfloor\frac{2}{3}n\rfloor}f_{n,i}$. 
This estimate relies on the fact that if $i,j,k$ are nonnegative
integers with $i+j+k\leq 2n$, then it follows that $\min\{i,j,k\}\leq \lfloor \frac{2n}{3}\rfloor$. But one can say something stronger, namely $i\leq \lfloor \frac{2n}{3}\rfloor$, $j\leq \lfloor \frac{2n-1}{3}\rfloor$ or
$k\leq \lfloor \frac{2n-2}{3}\rfloor$.
This observation gives a better bound  that is still based on the slice rank in the column labeled EG'.

In the table of Section~\ref{sec:capset} we have computed
the optimal value of $3\sum_{i=0}^{2n}f_{n,i}t_i$
rounded down to the nearest integer for $n\leq 20$.
This bound is an upper bound for  the cardinality of a cap set in $\F_3^n$.

Looking at optimal solutions for small $n$, we make the following conjecture:
\begin{conjecture}\label{conj:t}
The optimal solution of the linear program for $t_0,t_1,t_2,\dots,t_{2n}$ is as follows:
$$
\begin{cases}
\underbrace{1,1,\dots,1}_{\frac{2n-3}{3}},\frac{2}{3},\frac{1}{3},0,0,\dots & \mbox{if $n\equiv 0\bmod 3$}\\
\underbrace{1,1,\dots,1}_{\frac{2n-5}{3}},\frac{3}{4},\frac{1}{2},\frac{1}{4},0,0,\dots & \mbox{if $n\equiv 1\bmod 3$}\\
\underbrace{1,1,\dots,1}_{\frac{2n-7}{3}},\frac{4}{5},\frac{3}{5},\frac{2}{5},\frac{1}{5},0,0,\dots & \mbox{if $n\equiv 2\bmod 3$}
\end{cases}.
$$
\end{conjecture}
\section{Conclusion and further directions}
The $G$-stable rank is a new notion of rank for tensors. Up to a constant it is equal to the slice rank, but it is more refined in the sense that it can take non-integer values, and unlike the slice rank it is supermultiplicative with respect to vertical tensor products. As an illustration, we showed that the $G$-stable rank can be used to improve upper bounds for the cardinality of cap sets. A proof of Conjecture~\ref{conj:t} may lead to stronger asymptotic upper bounds
for the cap set problem. Numerical experiments
suggest an upper bound of the form $C\theta^n/\sqrt{n}$
for some constant $C$.

Besides algebraic applications of tensor decompositions  there are also 
many numerical applications such as psychometrics \cite{Tucker63,Tucker64,Tucker66,CarrolChang70,Harshman70} and chemometrics \cite{AppellofDavidson81}.
 For more details and references, see the survey article \cite{BaderKolda09} or the books \cite{Kroonenberg08,Landsberg12}. 
The formula (\ref{eq:rkG_C}) allows us to compute or approximate 
the $G$-stable rank for real or complex tensors using optimization. 
Future directions of research include algorithms for approximating the $G$-stable rank of a tensor, or to approximate a given tensors by tensors of low $G$-stable rank and apply these to such tasks as denoising, dimension reduction and tensor completion.

\end{document}